\documentclass[11pt]{article}
\usepackage{amsmath,amsthm,amssymb}
\usepackage{hyperref}
\usepackage{enumitem}
\usepackage{xcolor}
\usepackage[margin=2.5cm,a4paper]{geometry}
\definecolor{forestgreen(traditional)}{rgb}{0.0, 0.27, 0.13}
\definecolor{forestgreen(web)}{rgb}{0.13, 0.55, 0.13}
\definecolor{airforceblue}{rgb}{0.36, 0.54, 0.66}
\hypersetup{
    colorlinks,
    citecolor=forestgreen(web),
    linkcolor=airforceblue,
    urlcolor=magenta,
    pdftitle={A chain theorem for sequentially 3-rank-connected graphs with respect to vertex-minors},
    pdfauthor={Duksang Lee and Sang-il Oum}}

\newtheorem{theorem}{Theorem}[section]
\newenvironment{subproof}[1][\proofname]{%
  \begin{proof}[#1]%
}{%
  \end{proof}%
}
\newtheorem{lemma}[theorem]{Lemma}
\newtheorem{proposition}[theorem]{Proposition}
\theoremstyle{definition}
\newtheorem{claim}[theorem]{Claim}

\theoremstyle{remark}

\newcommand{\rnk}{\mathrm{rank}}
\newcommand{\GF}{\operatorname{GF}}

\usepackage{authblk}
\usepackage{lineno}
\title{A chain theorem for sequentially $3$-rank-connected graphs with respect to vertex-minors}
\author[2,1,3]{Duksang Lee\thanks{Supported by the Institute for Basic Science (IBS-R029-C1), the National Research Foundation of Korea(NRF) grant funded by the Korea government (NRF-2022M3J6A1063021),
and the KAIST Starting Fund (KAIST-G04220016).}}
\author[1,2]{Sang-il Oum\thanks{Supported by the Institute for Basic Science (IBS-R029-C1).}}
\affil[1]{Discrete Mathematics Group,
Institute for Basic Science (IBS),
Daejeon,~South~Korea}
\affil[2]{Department of Mathematical Sciences, KAIST, Daejeon, South~Korea}
\affil[3]{Department of Industrial and Systems Engineering, KAIST, Daejeon, South~Korea}
\affil[ ]{Email: \texttt{duksang@kaist.ac.kr}, \texttt{sangil@ibs.re.kr}}
\date{\today}							%

\begin{document}
\maketitle
\begin{abstract}
Tutte (1961) proved the chain theorem for simple $3$-connected graphs with respect to minors, which states that every simple $3$-connected graph $G$ has a simple $3$-connected minor with one edge fewer than $G$, unless $G$ is a wheel graph. Bouchet (1987) proved an analog for prime graphs with respect to vertex-minors. 
We present a chain theorem for higher connectivity with respect to vertex-minors, showing that
every sequentially $3$-rank-connected graph $G$ has a sequentially $3$-rank-connected vertex-minor with one vertex fewer than $G$, unless $|V(G)|\leq 12$.
\end{abstract}
\section{Introduction}\label{sec:intro}
Tutte~\cite{Tutte1961} proved the chain theorem for simple $3$-connected graphs with respect to minors, which states that every simple $3$-connected graph $G$ has a simple $3$-connected minor with one edge fewer than $G$, unless $G$ is a wheel graph. We will present a chain theorem for vertex-minors.

For a vertex $v$ of a graph $G$, the \emph{local complementation} at $v$ is an operation obtaining a new graph $G*v$ from $G$ by replacing the subgraph induced by the neighbors of $v$ with its complement graph. A graph $H$ is a \emph{vertex-minor} of $G$ if $H$ can be obtained from $G$ by a sequence of local complementations and vertex deletions.

For a graph $G$, the \emph{cut-rank} function $\rho_{G}$ is a function which maps a set $X$ of vertices of $G$ to the rank of a matrix over the binary field whose rows are labeled by $X$ and columns are labeled by $V(G)-X$, where the $(i,j)$-entry is $1$ if $i$ and $j$ are adjacent in $G$ and $0$ otherwise. A graph $G$ is \emph{prime} if there is no set $X$ of vertices of $G$ such that $|X|\geq 2$, $|V(G)-X|\geq 2$, and $\rho_{G}(X)\leq 1$.
Bouchet proved the following chain theorem for prime graphs with respect to vertex-minors. Later, Allys~\cite{Allys1994} proved a stronger theorem.

\begin{theorem}[Bouchet~{\cite[Theorem 3.2]{Bouchet1985}}]
\label{thm:bouchet}
Every prime graph $G$ has a prime vertex-minor $H$ with $|V(H)|=|V(G)|-1$, unless $|V(G)|\leq 5$.
\end{theorem}

A set $X$ of vertices of $G$ is \emph{sequential} in $G$ if there is an ordering $a_{1},\ldots,a_{k}$ of the vertices in $X$ such that $\rho_{G}(\{a_{1},\ldots,a_{i}\})\leq 2$ for each $1\leq i\leq k$. A graph $G$ is \emph{sequentially $3$-rank-connected} if it is prime and whenever $\rho_{G}(X)\leq 2$ for $X\subseteq V(G)$, either $X$ or $V(G)-X$ is sequential in $G$. 

Here is our chain theorem for sequentially $3$-rank-connected graphs with respect to vertex-minors.

\begin{theorem}
\label{thm:main}
Every sequentially $3$-rank-connected graph $G$ has a sequentially $3$-rank-connected vertex-minor $H$ with $|V(H)|=|V(G)|-1$, unless $|V(G)|\leq 12$.
\end{theorem}

Our theorem is motivated by the following theorem for sequentially $4$-connected matroids, proved by Geelen and Whittle. 
\begin{theorem}[Geelen and Whittle~{\cite[Theorem 1.2]{Geelen2001}}]
\label{thm:seq_matroid}
Every sequentially $4$-connected matroid~$M$ has a sequentially $4$-connected minor $N$ with $|E(N)|=|E(M)|-1$, unless $M$ is a wheel matroid or a whirl matroid.
\end{theorem}
Theorem~\ref{thm:seq_matroid} was motivated by the conjecture on the number of inequivalent representations over a fixed prime field. This conjecture was later proved by Geelen and Whittle~\cite{GW2013} by using a stronger version of Theorem~\ref{thm:seq_matroid} due to Oxley, Semple, and Whittle~\cite{OSW2012}. It would be interesting to see if this stronger version also has a vertex-minor analog.

Let us briefly sketch the proof of Theorem~\ref{thm:main}. The proof consists of three parts. In the first part, we prove it for $3$-rank-connected graphs that are prime graphs with no set $X$ such that $\rho_{G}(X)\leq 2$, $|X|>2$, and $|V(G)-X|>2$. 
The second part discusses internally $3$-rank-connected graphs that are not $3$-rank-connected. The last part considers sequentially $3$-rank-connected graphs that are not internally $3$-rank-connected. 

Essentially, the proof is based on the submodularity of the matrix rank function. We will also use Theorem~\ref{thm:bouchet}.
Proof ideas of some lemmas are from Geelen and Whittle~\cite{Geelen2001}. We will also use triplets introduced by Oum~\cite{Oum2020}. 

Our paper is organized as follows. In Section~\ref{sec:prim}, we review vertex-minors and several inequalities for cut-rank functions. In Section~\ref{sec:def_sequential}, we prove elementary lemmas on sequential sets and sequentially $3$-rank-connected graphs. In Section~\ref{sec:3rank}, we prove the main theorem for $3$-rank-connected graphs. In Section~\ref{sec:internal}, we prove our theorem for internally $3$-rank-connected graphs. In Section~\ref{sec:sequential}, we conclude the proof by dealing with sequentially $3$-rank-connected graphs which are not internally $3$-rank-connected.

\section{Preliminaries}
\label{sec:prim}
A graph is \emph{simple} if it has no loops and parallel edges. In this paper, all graphs are finite and simple. For a graph $G$ and a vertex $v$, let $N_{G}(v)$ be the set of vertices adjacent to $v$ in $G$. For a graph $G$ and a subset $X$ of $V(G)$, let $G[X]$ be the subgraph of $G$ induced on $X$. 

\paragraph{Vertex-minors}
For a graph $G$ and a vertex $v$ of $G$, let $G*v$ be the graph obtained by replacing $G[N_{G}(v)]$ with its complement. The operation obtaining $G*v$ from $G$ is called the \emph{local complementation} at $v$. A graph $H$ is \emph{locally equivalent} to $G$ if $H$ can be obtained from $G$ by a sequence of local complementations. A graph $H$ is a \emph{vertex-minor} of a graph $G$ if $H$ can be obtained from~$G$ by applying local complementations and deleting vertices.

For an edge $uv$ of a graph $G$, let $G\wedge uv=G*u*v*u$. Then $G\wedge uv$ is obtained from $G$ by \emph{pivoting} $uv$. The graph $G\wedge uv$ is well defined since $G*u*v*u=G*v*u*v$~{\cite[Corollary 2.2]{Oum2005}}. 

\begin{lemma}[see Oum~{\cite{Oum2005}}]
\label{lem:pivot_eq}
Let $G$ be a graph and $v$ be a vertex of $G$. If $x,y\in N_{G}(v)$, then $(G\wedge vx)\setminus v$ is locally equivalent to $(G\wedge vy)\setminus v$.
\end{lemma}

By Lemma~\ref{lem:pivot_eq}, we write $G/v$ to denote $G\wedge uv\setminus v$ for a neighbor $u$ of $v$ in $G$ because we are only interested in graphs up to local equivalence. 

\begin{lemma}[Geelen and Oum~{\cite[Lemma 3.1]{Geelen2009}}]
\label{lem:onetoone}
Let $G$ be a graph and $v$ and $w$ be vertices of~$G$. Then the following hold.
\begin{enumerate}[label=\rm(\arabic*)]
\item If $v\neq w$ and $vw\notin E(G)$, then $(G*w)\setminus v$, $(G*w*v)\setminus v$, and $(G*w)/v$ are locally equivalent to $G\setminus v$, $G*v\setminus v$, and $G/v$ respectively.
\item If $v\neq w$ and $vw\in E(G)$, then $(G*w)\setminus v$, $(G*w*v)\setminus v$, and $(G*w)/v$ are locally equivalent to $G\setminus v$, $G/v$, and $(G*v)\setminus v$ respectively.
\item If $v=w$, then $(G*w)\setminus v$, $(G*w*v)\setminus v$, and $(G*w)/v$ are locally equivalent to $G*v\setminus v$, $G\setminus v$, and $G/v$ respectively.
\end{enumerate}
\end{lemma}

Lemma~\ref{lem:onetoone} implies the following lemma, which was first proved by Bouchet.

\begin{lemma}[Bouchet~{\cite[Corollary 9.2]{Bouchet1988}}]
\label{lem:elt_minor}
Let $H$ be a vertex-minor of a graph $G$ such that $V(H)=V(G)-\{v\}$ for a vertex $v$ of $G$. Then $H$ is locally equivalent to one of $G\setminus v$, $G*v\setminus v$, and $G/v$.
\end{lemma}

\paragraph{Cut-rank function and rank-connectivity}
For an $X\times Y$-matrix $A$ and $I\subseteq X$, $J\subseteq Y$, let $A[I,J]$ be an $I\times J$-submatrix of $A$. Let $A_{G}$ be the adjacency matrix of a graph $G$ over the binary field $\GF(2)$. The \emph{cut-rank} $\rho_{G}(X)$ of a subset $X$ of $V(G)$ is defined by
\[
\rho_{G}(X)=\rnk(A_{G}[X,V(G)-X]).
\]
It is trivial to check that $\rho_{G}(X)=\rho_{G}(V(G)-X)$. For disjoint sets $X$, $Y$ of a graph $G$, let $\rho_{G}(X,Y)=\rnk(A_{G}[X,Y])$. 
A graph $G$ is \emph{$k$-rank-connected} if there is no partition $(A,B)$ of $V(G)$ such that $|A|,|B|>\rho_{G}(A)$ and $\rho_{G}(A)<k$. A graph is \emph{prime} if it is $2$-rank-connected. Observe that $1$-rank-connected graphs are connected graphs.

\begin{lemma}
\label{lem:deg3} If $G$ is a $3$-rank-connected graph with at least $6$ vertices, then $\deg_{G}(v)\geq 3$ for each $v\in V(G)$.
\end{lemma}
\begin{proof}
Suppose that $\deg_{G}(v)\leq 2$.
Let $X$ be the set of neighbors of $v$. Then $\rho_{G}(X\cup\{v\})\leq |X|\leq 2$. However, $\rho_{G}(X\cup\{v\})<|X\cup\{v\}|$ and $2<|V(G)-(X\cup\{v\})|$, contradicting assumption that $G$ is $3$-rank-connected.
\end{proof}

\begin{lemma}[Oum~{\cite[Proposition 2.4]{Oum2020}}]
\label{lem:krank}
Let $k$ be a positive integer.
If a graph $G$ is $k$-rank-connected and $|V(G)|\geq 2k$, then for each $v\in V(G)$, the graph $G\setminus v$ is $(k-1)$-rank-connected.
\end{lemma}

\begin{lemma}
\label{lem:kconn}
Let $k$ be a positive integer.
A $k$-rank-connected graph with $|V(G)|\geq 2k$ is $k$-connected.
\end{lemma}
\begin{proof}
We use induction on $k$. Let $G$ be a $k$-rank-connected graph with $\lvert V(G)\rvert\ge 2k$. We may assume that $k>1$. 
Let $X$ be a subset of $V(G)$ with $\lvert X\rvert<k$. 
It is enough to prove that $G\setminus X$ is connected.
Since $G$ is $1$-rank-connected, $G$ is connected and therefore we may assume that $X$ is nonempty.
Let $v$ be a vertex in~$X$.
By applying Lemma~\ref{lem:krank} and the induction hypothesis, $G\setminus v$ is $(k-1)$-connected and therefore $(G\setminus v)\setminus (X-\{v\})=G\setminus X$ is connected. 
\end{proof}

The following lemmas give properties of the matrix rank function and the cut-rank function.

\begin{lemma}[see Oum~{\cite[Proposition 2.6]{Oum2005}}]
\label{lem:local}
If a graph $G'$ is locally equivalent to a graph $G$, then $\rho_{G}(X)=\rho_{G'}(X)$ for each $X\subseteq V(G)$.
\end{lemma}

\begin{lemma}
\label{lem:delrank}
Let $G$ be a graph and $v$ be a vertex of $G$. For a subset $X$ of $V(G)-\{v\}$, we have
\begin{enumerate}[label=\rm(\roman*)]
\item\label{item:2.4i} $\rho_{G\setminus v}(X)+1\geq\rho_{G}(X)\geq\rho_{G\setminus v}(X)$.
\item\label{item:2.4ii} $\rho_{G\setminus v}(X)+1\geq\rho_{G}(X\cup\{v\})\geq\rho_{G\setminus v}(X)$.
\end{enumerate}
\end{lemma}
\begin{proof}
Observe that removing a row or a column of a matrix decreases the rank by at most~$1$.
\end{proof}

\begin{lemma}[see Truemper~\cite{Truemper1985}]
\label{lem:rank_subeq}
Let $A$ be an $X\times Y$-matrix. For sets $X_{1}, X_{2}\subseteq X$ and $Y_{1}, Y_{2}\subseteq Y$,
\[
\rnk(A[X_{1},Y_{1}])+\rnk(A[X_{2},Y_{2}])\geq\rnk(A[X_{1}\cap X_{2},Y_{1}\cup Y_{2}])+\rnk(A[X_{1}\cup X_{2},Y_{1}\cap Y_{2}]).
\]
\end{lemma}

Lemma~\ref{lem:rank_subeq} implies the following seven lemmas.

\begin{lemma}[see Oum~{\cite[Corollary 4.2]{Oum2005}}]
\label{lem:subeq}
Let $G$ be a graph and let $X$, $Y$ be subsets of $V(G)$. Then,
\[
\rho_{G}(X)+\rho_{G}(Y)\geq\rho_{G}(X\cap Y)+\rho_{G}(X\cup Y).
\]
\end{lemma}

\begin{lemma}
\label{lem:subeq_minus}
Let $G$ be a graph and $X$ and $Y$ be subsets of $V(G)$. Then,
\[
\rho_{G}(X)+\rho_{G}(Y)\geq\rho_{G}(Y-X)+\rho_{G}(X-Y).
\]
\end{lemma}
\begin{proof}
Apply Lemma~\ref{lem:subeq} with $X$ and $V(G)-Y$.
\end{proof}

\begin{lemma}[Oum~{\cite[Lemma 2.3]{Oum2020}}] 
\label{lem:subtool}
Let $G$ be a graph and $v$ be a vertex of $G$. Let $X$ and $Y$ be subsets of $V(G)-\{v\}$. Then, the following hold.
\begin{enumerate}[label=\rm(S\arabic*)]
\item\label{item:s1} $\rho_{G\setminus v}(X)+\rho_{G}(Y\cup\{v\})\geq\rho_{G\setminus v}(X\cap Y)+\rho_{G}(X\cup Y\cup\{v\})$.
\item\label{item:s2} $\rho_{G\setminus v}(X)+\rho_{G}(Y)\geq\rho_{G}(X\cap Y)+\rho_{G\setminus v}(X\cup Y)$.
\end{enumerate}
\end{lemma}

\begin{lemma}
\label{lem:cor_s2}
Let $G$ be a graph and $v$ be a vertex of $G$. Let $X$, $Y$ be subsets of $V(G\setminus v)$. If $X\subseteq Y$ and $\rho_{G\setminus v}(Y)\geq\rho_{G}(Y)$, then $\rho_{G\setminus v}(X)=\rho_{G}(X)$.
\end{lemma}
\begin{proof}
By \ref{item:s2} of Lemma~\ref{lem:subtool}, 
\[
\rho_{G\setminus v}(X)+\rho_{G}(Y)\geq\rho_{G\setminus v}(Y)+\rho_{G}(X).
\]
Therefore, by Lemma~\ref{lem:delrank}(i), $0\leq\rho_{G}(X)-\rho_{G\setminus v}(X)\leq\rho_{G}(Y)-\rho_{G\setminus v}(Y)\leq0$. So we conclude that $\rho_{G\setminus v}(X)=\rho_{G}(X)$.
\end{proof}

\begin{lemma}
\label{lem:cor_s1}
Let $G$ be a graph and $v$ be a vertex of $G$. Let $X$, $Y$ be subsets of $V(G)$. If $v\in Y\subseteq X$ and $\rho_{G\setminus v}(Y-\{v\})\geq\rho_{G}(Y)$, then $\rho_{G\setminus v}(X-\{v\})=\rho_{G}(X)$.
\end{lemma}
\begin{proof}
We apply Lemma~\ref{lem:cor_s2} for $V(G)-X$ and $V(G)-Y$.
\end{proof}
\begin{lemma}
\label{lem:subtool_minus}
Let $G$ be a graph and $v$ be a vertex of $G$. Let $X$ and $Y$ be subsets of $V(G)-\{v\}$. Then, 
\[
\rho_{G\setminus v}(X)+\rho_{G}(Y\cup\{v\})\geq\rho_{G\setminus v}(Y-X)+\rho_{G}(X-Y).
\]
\end{lemma}
\begin{proof}
Apply~\ref{item:s1} of Lemma~\ref{lem:subtool} with $V(G)-(X\cup\{v\})$ and $Y$. 
\end{proof}

\begin{lemma}[Oum~{\cite[Lemma 2.2]{Oum2020}}]
\label{lem:sub_eq_AB}
Let $G$ be a graph and $a$, $b$ be distinct vertices of $G$. Let $A\subseteq V(G)-\{a\}$ and $B\subseteq V(G)-\{b\}$. Then, the following hold.
\begin{enumerate}[label=\rm(A\arabic*)]
\item\label{item:ab1} If $b\notin A$ and $a\notin B$, then $\rho_{G}(A\cap B)+\rho_{G\setminus a\setminus b}(A\cup B)\leq\rho_{G\setminus a}(A)+\rho_{G\setminus b}(B)$.
\item\label{item:ab2} If $b\in A$ and $a\notin B$, then $\rho_{G\setminus b}(A\cap B)+\rho_{G\setminus a}(A\cup B)\leq\rho_{G\setminus a}(A)+\rho_{G\setminus b}(B)$.
\item\label{item:ab3} If $b\in A$ and $a\in B$, then $\rho_{G\setminus a\setminus b}(A\cap B)+\rho_{G}(A\cup B)\leq\rho_{G\setminus a}(A)+\rho_{G\setminus b}(B)$.
\end{enumerate}
\end{lemma}

\begin{lemma}[Oum~{\cite[Proposition 4.3]{Oum2005}}]
\label{lem:matrix_local}
Let $G$ be a graph and $x$ be a vertex of $G$. For a subset $X$ of $V(G)-\{x\}$, the following hold.
\begin{enumerate}[label=\rm(\arabic*)]
\item $\rho_{G*x\setminus x}(X)=\rnk
\begin{pmatrix}
1 & A_{G}[\{x\},V(G)-(X\cup\{x\})] \\ 
A_{G}[X,\{x\}] & A_{G}[X,V(G)-(X\cup\{x\})]
\end{pmatrix}-1$.
\item$\rho_{G/x}(X)=\rnk
\begin{pmatrix}
0 & A_{G}[\{x\},V(G)-(X\cup\{x\})] \\ 
A_{G}[X,\{x\}] & A_{G}[X,V(G)-(X\cup\{x\})]
\end{pmatrix}-1$.
\end{enumerate}
\end{lemma}

From Lemma~\ref{lem:matrix_local}, we deduce the following lemma.

\begin{lemma}
\label{lem:local_or_pivot}
Let $G$ be a graph and $x\in V(G)$. Let $C$ be a subset of $V(G)-\{x\}$ such that $\rho_{G\setminus x}(C)=\rho_{G}(C)$. Then $\rho_{G*x\setminus x}(C)=\rho_{G}(C\cup\{x\})-1$ or $\rho_{G/x}(C)=\rho_{G}(C\cup\{x\})-1$.
\end{lemma}
\begin{proof}
Let $D=V(G)-(C\cup\{x\})$. Since $\rho_{G\setminus x}(C)=\rho_{G}(C)$, a column vector $A_{G}[C,\{x\}]$ is in the column space of $A_{G}[C,D]$. Then let $A'$ and $A''$ be matrices over $\GF(2)$ such that
\[
A'=\begin{pmatrix}
1 & A_{G}[\{x\},D] \\ 
A_{G}[C,\{x\}] & A_{G}[C,D]
\end{pmatrix}
\text{ and }
A''=\begin{pmatrix}
0 & A_{G}[\{x\},D] \\ 
A_{G}[C,\{x\}] & A_{G}[C,D]
\end{pmatrix}.
\]
Then $\rnk(A')=\rho_{G}(C\cup\{x\})$ or $\rnk(A'')=\rho_{G}(C\cup\{x\})$ and therefore, by Lemma~\ref{lem:matrix_local}, we have 
$\rho_{G*x\setminus x}(C)=\rnk(A')-1=\rho_{G}(C\cup\{x\})-1$ or $\rho_{G/x}(C)=\rnk(A'')-1=\rho_{G}(C\cup\{x\})-1$.
\end{proof}

\begin{lemma}[Oum~{\cite[Lemma 4.4]{Oum2005}}]
\label{lem:pivot_subeq}
Let $G$ be a graph and $x$ be a vertex of $G$. Let $(X_{1},Y_{1})$ and $(X_{2},Y_{2})$ be partitions of $V(G)-\{x\}$. Then the following hold:
\begin{enumerate}[label=\rm(P\arabic*)]
\item\label{item:p1} $\rho_{G\setminus x}(X_{1})+\rho_{G*x\setminus x}(X_{2})\geq\rho_{G}(X_{1}\cap X_{2})+\rho_{G}(Y_{1}\cap Y_{2})-1$.
\item\label{item:p2} $\rho_{G\setminus x}(X_{1})+\rho_{G/x}(X_{2})\geq\rho_{G}(X_{1}\cap X_{2})+\rho_{G}(Y_{1}\cap Y_{2})-1$.
\end{enumerate}
\end{lemma}
The following lemma is an easy consequence of Lemmas~\ref{lem:local} and~\ref{lem:pivot_subeq}.

\begin{lemma}
\label{lem:pivot_subeq2}
Let $G$ be a graph and $x$ be a vertex of $G$. Let $(X_{1},Y_{1})$ and $(X_{2},Y_{2})$ be partitions of $V(G)-\{x\}$. Then,
\[
\rho_{G*x\setminus x}(X_{1})+\rho_{G/x}(X_{2})\geq\rho_{G}(X_{1}\cap X_{2})+\rho_{G}(Y_{1}\cap Y_{2})-1.
\]
\end{lemma}

\section{Sequentially 3-rank-connected graphs}
\label{sec:def_sequential}
Let us recall the definition of sequentially $3$-rank-connected graphs introduced in Section~\ref{sec:intro}.
A subset $A$ of $V(G)$ is \emph{sequential} in a graph $G$ if there is an ordering $a_{1},\ldots, a_{|A|}$ of the elements of $A$ such that $\rho_{G}(\{a_{1},\ldots, a_{i}\})\leq 2$ for each $1\leq i\leq |A|$.
A graph $G$ is \emph{sequentially $3$-rank-connected} if it is prime and for each subset $X$ of $V(G)$ with $\rho_{G}(X)\leq 2$, we have that $X$ or $V(G)-X$ is sequential in $G$.

We now present basic lemmas on sequential sets and sequentially $3$-rank-connected graphs.
\begin{lemma}
\label{lem:basic_sequential}
Let $G$ be a graph and $A$ be a subset of $V(G)$. Let $t$ be a vertex of $G$ such that $\rho_{G}(A\cup\{t\})=\rho_{G}(A)$. Then $A\cup\{t\}$ is sequential in $G$ if and only if $A$ is sequential in $G$.
\end{lemma}
\begin{proof}
We may assume that $t\notin A$.
The backward direction is obvious. So it is enough to show the forward direction. 

Since $A\cup\{t\}$ is sequential in $G$, there is an ordering $a_{1},\ldots,a_{m}$ of the elements of $A\cup\{t\}$ such that $m=|A\cup\{t\}|$ and $\rho_{G}(\{a_{1},\ldots,a_{i}\})\leq 2$ for each $1\leq i\leq m$. Let $1\leq j\leq m$ be an index such that $a_{j}=t$. Then for each $j+1\leq i\leq m$, by Lemma~\ref{lem:subeq}, we have
\[
\rho_{G}(\{a_{1},\ldots,a_{i}\})+\rho_{G}(A)\geq\rho_{G}(A\cup\{t\})+\rho_{G}(\{a_{1},\ldots,a_{i}\}-\{t\}),
\]
and therefore $\rho_{G}(\{a_{1},\ldots,a_{i}\}-\{t\})\leq\rho_{G}(\{a_{1},\ldots,a_{i}\})$.
For each $1\leq i\leq m-1$, let
\[
a_{i}'=
\begin{cases}
a_{i} & \text{if $i<j$,} \\
a_{i+1} & \text{if $i\geq j$.}
\end{cases}
\]
Hence, by above inequality, $A$ is sequential in $G$ because $a_{1}',\ldots,a_{m-1}'$ is a desired ordering of the elements of $A$.
\end{proof}

\begin{lemma}
\label{lem:contain_triple}
Let $G$ be a prime graph that is not sequentially $3$-rank-connected and let $T_{1}, \ldots, T_{n}$ be pairwise disjoint $3$-element subsets of $V(G)$ such that $\rho_{G}(T_{i})=2$ for each $1\leq i\leq n$. Then there exists a subset $A$ of $V(G)$ such that $\rho_{G}(A)\leq 2$, neither $A$ nor $V(G)-A$ is sequential in~$G$, 
and for each $1\leq i\leq n$, we have that $T_{i}\subseteq A$ or $T_{i}\subseteq V(G)-A$.
\end{lemma}
\begin{proof}
We proceed by induction on $n$.
Since $G$ is prime and not sequentially $3$-rank-connected, there is a subset $A$ of $V(G)$ such that $\rho_{G}(A)\leq 2$, and neither $A$ nor $V(G)-A$ is sequential in~$G$. So we can assume that $n\geq 1$. By the induction hypothesis, there is a subset $A'$ of $V(G)$ such that $\rho_{G}(A')\leq 2$, and neither $A'$ nor $V(G)-A'$ is sequential in $G$, and for each $1\leq i\leq n-1$, either $T_{i}\subseteq A'$ or $T_{i}\subseteq V(G)-A'$. Let $B'=V(G)-A'$. We may assume that $A'\cap T_{n}\neq\emptyset$ and $B'\cap T_{n}\neq\emptyset$. Then, by symmetry, we can assume that $|A'\cap T_{n}|=2$ and let $x$ be the element of $B'\cap T_{n}$. Since $|T_{n}-\{x\}|=2$ and $G$ is prime, we have $\rho_{G}(T_{n}-\{x\})=2=\rho_{G}(T_{n})$. Then, by Lemma~\ref{lem:subeq},
\[
\rho_{G}(A')+2=\rho_{G}(A')+\rho_{G}(T_{n})\geq\rho_{G}(A'\cup\{x\})+\rho_{G}(T_{n}-\{x\})=\rho_{G}(A'\cup\{x\})+2.
\] 

Hence $\rho_{G}(A'\cup\{x\})\leq\rho_{G}(A')\leq 2$. Since $V(G)-A'$ is not sequential in $G$, 
$|V(G)-A'|\geq 4$ and so $|V(G)-(A'\cup\{x\})|\geq 3$. Hence $\rho_{G}(A')=\rho_{G}(A'\cup\{x\})=2$ because $G$ is prime. Hence, by Lemma~\ref{lem:basic_sequential}, neither $A'\cup\{x\}$ nor $V(G)-(A'\cup\{x\})$ is sequential in $G$.  

For each $1\leq i\leq n-1$, we have $x\notin T_{i}$ because $T_{n}$ and $T_{i}$ are disjoint. Therefore, $T_{i}\subseteq A'\cup\{x\}$ or $T_{i}\subseteq V(G)-(A'\cup\{x\})$ for each $1\leq i\leq n$.
\end{proof}

\section{Treating $3$-rank-connected graphs}
\label{sec:3rank}
In this section, we prove Theorem~\ref{thm:main} for $3$-rank-connected graphs.

The following lemma shows that every vertex-minor of a $3$-rank-connected graph $G$ with one vertex fewer than~$G$ is prime.

\begin{lemma}
\label{lem:3toprime}
Let $G$ be a $3$-rank-connected graph with $|V(G)|\geq6$ and $x$ be a vertex of $G$. Then all of $G\setminus x$, $G*x\setminus x$, and $G/x$ are prime.
\end{lemma}
\begin{proof}
By Lemma~\ref{lem:local}, it is enough to show that $G\setminus x$ is prime. This is implied by Lemma~\ref{lem:krank}.
\end{proof}

A graph $G$ is \emph{weakly $3$-rank-connected} if $G$ is prime and $V(G)$ has no subset $X$ such that $|X|\geq 5$, $|V(G)-X|\geq 5$, and $\rho_{G}(X)\leq 2$.
The following lemma can be deduced easily from {\cite[Proposition 2.6]{Oum2020}} and Lemma~\ref{lem:onetoone}. 
\begin{lemma}[Oum~\cite{Oum2020}]
\label{lem:weakly}
Let $G$ be a $3$-rank-connected graph with $|V(G)|\geq 6$ and $x$ be a vertex of~$G$. Then at least two of $G\setminus x$, $G*x\setminus x$, and $G/x$ are weakly $3$-rank-connected. 
\end{lemma}

\begin{lemma}
\label{lem:weak_conn}
Let $G$ be a $3$-rank-connected graph with $|V(G)|\geq 6$ and let $S=\{v_{1},\cdots,v_{t}\}$ be the set of all vertices~$x$ of $G$ such that $G\setminus x$ is not weakly $3$-rank-connected. Let $G'=G*v_{1}*\cdots*v_{t}$. Then $G'\setminus v$ is weakly $3$-rank-connected for every vertex $v$ of $G'$.
\end{lemma}
\begin{proof}
If $v\notin S$, then $G'\setminus v=(G\setminus v)*v_{1}*\cdots*v_{t}$ and so $G'\setminus v$ is weakly $3$-rank-connected. If $v=v_{i}$ for some $1\leq i\leq t$, then by Lemma~\ref{lem:weakly}, $G*v\setminus v$ is weakly $3$-rank-connected. Since $G'\setminus v=(G*v\setminus v)*v_{1}*\cdots*v_{i-1}*v_{i+1}*\cdots*v_{t}$ is locally equivalent to $G*v\setminus v$, we deduce that $G'\setminus v$ is weakly $3$-rank-connected.
\end{proof}

\begin{lemma}
\label{lem:intersection2}
Let $G$ be a $3$-rank-connected graph and $x$ be a vertex of $G$. Let $P$ be a $4$-element subset of $V(G)-\{x\}$ such that $\rho_{G\setminus x}(P)\leq 2$ and $(A,B)$ be a partition of $V(G)-\{x\}$ such that $|A|,|B|\geq 4$ and $\rho_{H}(A)\leq 2$ for some $H\in\{G*x\setminus x, G/x\}$.
Then $|A\cap P|=|B\cap P|=2$.
\end{lemma}
\begin{proof}
Suppose that $|A\cap P|\neq |B\cap P|$. We may assume that $|A\cap P|>|B\cap P|$.
Since $\rho_{G\setminus x}(P)\leq 2$ and $\rho_{H}(A)\leq 2$, by~\ref{item:p1} and~\ref{item:p2} of Lemma~\ref{lem:pivot_subeq}, we have 
\[
4\geq\rho_{G\setminus x}(P)+\rho_{H}(A)\geq\rho_{G}(A\cap P)+\rho_{G}(B-P)-1.
\]

Since $|A\cap P|>2$ and $G$ is $3$-rank-connected, $\rho_{G}(A\cap P)>2$. Hence $\rho_{G}(B-P)\leq 2$. Since $G$ is $3$-rank-connected, $|B-P|\leq 2$, which implies that $|B\cap P|\geq 2$, contradicting the fact that $|P|=4$.
\end{proof}

A $4$-element subset $P$ of $V(G)$ is a \emph{quad} of $G$ if $\rho_{G}(P)=2$ and $\rho_{G}(P-\{x\})=3$ for each $x\in P$. 

\begin{lemma}
\label{lem:quad}
Let $G$ be a prime graph and $A$ be a subset of $V(G)$ such that $\rho_{G}(A)=2$ and $|A|\leq 4$. Then $A$ is a quad of $G$ or $A$ is sequential in $G$.
\end{lemma}
\begin{proof}
Suppose that $A$ is not sequential in $G$. Then $|A|=4$ and $\rho_{G}(T)=3$ for each $3$-element subset $T$ of $A$. Therefore, $A$ is a quad of $G$.
\end{proof}

Our key ingredient of this section is Proposition~\ref{prop:key}, which states that it is sufficient to identify a set $\{t_1, t_2, t_3\}$ of three vertices  and a quad $Q_i$ from $G\setminus t_i$ for each $i\in\{1,2,3\}$ that satisfy the following conditions:
\begin{enumerate}
 \item[(1)] $G\setminus t_i$ is weakly $3$-rank-connected for each $i\in\{1,2,3\}$. 
 \item[(2)] $Q_1\cap Q_2=\{t_{3}\}$, $Q_2\cap Q_3=\{t_{1}\}$, and $Q_3\cap Q_1=\{t_{2}\}$.
\end{enumerate}
The remainder of this section will focus on identifying these three vertices and quads.

\begin{proposition}
\label{prop:key}
Let $t_{1}$, $t_{2}$, and $t_{3}$ be distinct vertices of a $3$-rank-connected graph $G$ such that $G\setminus t_{1}$, $G\setminus t_{2}$, and $G\setminus t_{3}$ are weakly $3$-rank-connected. For each $i\in\{1,2,3\}$, let $Q_{i}$ be a quad of $G\setminus t_{i}$. 
If $Q_{1}\cap Q_{2}=\{t_{3}\}$, $Q_{2}\cap Q_{3}=\{t_{1}\}$, and $Q_{3}\cap Q_{1}=\{t_{2}\}$,
then for each $i\in\{1,2,3\}$, either $G*t_{i}\setminus t_{i}$ or $G/t_{i}$ is sequentially $3$-rank-connected.
\end{proposition}
\begin{proof}
Since $|V(G)|\geq |Q_{1}\cup Q_{2}|=7$, by Lemma~\ref{lem:3toprime}, all of $G\setminus v$, $G*v\setminus v$, and $G/v$ are prime for each vertex $v$ of $G$.
Observe that $\{t_2,t_3\}\subseteq Q_1$, $\{t_1,t_3\}\subseteq Q_2$, and $\{t_1,t_2\}\subseteq Q_3$. For each $i\in\{1,2,3\}$, let $a_i$ and $b_i$ be two distinct vertices of $Q_{i}-\{t_1,t_2,t_3\}$. 

Suppose that neither $G*t_{1}\setminus t_{1}$ nor $G/t_{1}$ is sequentially $3$-rank-connected. 
Let us first show that $\rho_{G\setminus t_{1}}(\{t_2,a_3,b_3\})=3$. Since $G\setminus t_{1}$ is prime, $\rho_{G\setminus t_{1}}(\{a_{3},b_{3}\})=2=\rho_{G\setminus t_{1}}(\{t_{2},t_{3},a_{1},b_{1}\})$. 
By Lemma~\ref{lem:subeq_minus}, 
\[
\rho_{G\setminus t_{1}}(\{t_{2},a_{3},b_{3}\})+\rho_{G\setminus t_{1}}(\{t_{2},t_{3},a_{1},b_{1}\})\geq\rho_{G\setminus t_{1}}(\{a_{3},b_{3}\})+\rho_{G\setminus t_{1}}(\{t_{3},a_{1},b_{1}\}),
\]
and therefore $\rho_{G\setminus t_{1}}(\{t_{2},a_{3},b_{3}\})\geq\rho_{G\setminus t_{1}}(\{t_{3},a_{1},b_{1}\})$. Since $Q_{1}=\{t_{2},t_{3},a_{1},b_{1}\}$ is a quad of $G\setminus t_{1}$, $\rho_{G\setminus t_{1}}(\{t_{3},a_{1},b_{1}\})=3$. Therefore $\rho_{G\setminus t_{1}}(\{t_{2},a_{3},b_{3}\})=3$ and, by symmetry, $\rho_{G\setminus t_{1}}(\{t_{3},a_{2},b_{2}\})=3$.

Since $3=\rho_{G\setminus t_{1}}(\{t_{2},a_{3},b_{3}\})\leq\rho_{G}(\{t_{2},a_{3},b_{3}\})\leq 3$, we have $\rho_{G}(\{t_{2},a_{3},b_{3}\})=3$. Since $Q_{3}=\{t_{1},t_{2},a_{3},b_{3}\}$ is a quad of $G\setminus t_{3}$ and $G$ is $3$-rank-connected, we observe that $3\leq\rho_{G}(\{t_{1},t_{2},a_{3},b_{3}\})\leq 1+\rho_{G\setminus t_{3}}(\{t_{1},t_{2},a_{3},b_{3}\})=3$ and therefore $\rho_{G}(\{t_{1},t_{2},a_{3},b_{3}\})=3$. Similarly, $\rho_{G}(\{t_{3},a_{2},b_{2}\})=\rho_{G}(\{t_{1},t_{3},a_{2},b_{2}\})=~3$.
Therefore, by Lemma~\ref{lem:local_or_pivot}, the following hold.
\begin{enumerate}[label=\rm(R\arabic*)]
\item\label{item:rk1} $\rho_{G*t_{1}\setminus t_{1}}(\{t_{2},a_{3},b_{3}\})=2$ or $\rho_{G/t_{1}}(\{t_{2},a_{3},b_{3}\})=2$.
\item\label{item:rk2} $\rho_{G*t_{1}\setminus t_{1}}(\{t_{3},a_{2},b_{2}\})=2$ or $\rho_{G/t_{1}}(\{t_{3},a_{2},b_{2}\})=2$.
\end{enumerate}

Since $G$ is $3$-rank connected, $\rho_{G}(\{t_{2},a_{3},b_{3}\}), \rho_{G}(\{t_{3},a_{2},b_{2}\})\geq 3$. So by Lemma~\ref{lem:pivot_subeq2}, 
\[
\rho_{G*t_{1}\setminus t_{1}}(\{t_{2},a_{3},b_{3}\})+\rho_{G/t_{1}}(V(G\setminus t_{1})-\{t_{3},a_{2},b_{2}\})=\rho_{G}(\{t_{2},a_{3},b_{3}\})+\rho_{G}(\{t_{3},a_{2},b_{2}\})-1\geq 5.
\]
Hence, $\rho_{G*t_{1}\setminus t_{1}}(\{t_{2},a_{3},b_{3}\})+\rho_{G/t_{1}}(\{t_{3},a_{2},b_{2}\})\geq 5$ and similarly, 
\[
\rho_{G*t_{1}\setminus t_{1}}(\{t_{3},a_{2},b_{2}\})+\rho_{G/t_{1}}(\{t_{2},a_{3},b_{3}\})\geq~5.
\] 
Therefore, by \ref{item:rk1} and \ref{item:rk2}, either
\begin{enumerate}[label=\rm(\alph*)]
\item\label{item:a} $\rho_{G*t_{1}\setminus t_{1}}(\{t_{2},a_{3},b_{3}\})=\rho_{G*t_{1}\setminus t_{1}}(\{t_{3},a_{2},b_{2}\})=2$, or
\item\label{item:b} $\rho_{G/t_{1}}(\{t_{2},a_{3},b_{3}\})=\rho_{G/t_{1}}(\{t_{3},a_{2},b_{2}\})=2$.
\end{enumerate}

By Lemma~\ref{lem:onetoone}, we may assume \ref{item:a}, because otherwise we can choose a neighbor $y\notin\{t_{2},t_{3}\}$ of $t_{1}$ in $G$ by Lemma~\ref{lem:deg3} and replace $G$ by $G*y$. By Lemma~\ref{lem:contain_triple}, there is a subset $A$ of $V(G*t_{1}\setminus t_{1})$ such that 
\begin{itemize}
\item $\rho_{G*t_{1}\setminus t_{1}}(A)\leq 2$, 
\item neither $A$ nor $V(G*t_{1}\setminus t_{1})-A$ is sequential in $G*t_{1}\setminus t_{1}$, 
\item $\{t_{2},a_{3},b_{3}\}\subseteq A$ or $\{t_{2},a_{3},b_{3}\}\subseteq V(G*t_{1}\setminus t_{1})-A$, and 
\item $\{t_{3},a_{2},b_{2}\}\subseteq A$ or $\{t_{3},a_{2},b_{2}\}\subseteq V(G*t_{1}\setminus t_{1})-A$. 
\end{itemize}
We may assume that $\{t_{2},a_{3},b_{3}\}\subseteq A$ by replacing $A$ with $V(G*t_{1}\setminus t_{1})-A$ if necessary. 
Let $B=V(G*t_{1}\setminus t_{1})-A$. 

Suppose that $\{t_{3},a_{2},b_{2}\}\subseteq A$. Observe that $\rho_{G}(A)\leq\rho_{G*t_{1}\setminus t_{1}}(A)+1\leq 3$.
Since $\{t_{1},t_{2},a_{3},b_{3}\}$ is a quad of $G\setminus t_{3}$, by~\ref{item:s1} of Lemma~\ref{lem:subtool}, 
\begin{align*}
3+2&\geq\rho_{G}(A)+\rho_{G\setminus t_{3}}(\{t_{1},t_{2},a_{3},b_{3}\})=\rho_{G}((A-\{t_{3}\})\cup\{t_{3}\})+\rho_{G\setminus t_{3}}(\{t_{1},t_{2},a_{3},b_{3}\}) \\ 
&\geq\rho_{G\setminus t_{3}}((A-\{t_{3}\})\cap\{t_{1},t_{2},a_{3},b_{3}\})+\rho_{G}((A-\{t_{3}\})\cup\{t_{1},t_{2},t_{3},a_{3},b_{3}\}) \\
&=\rho_{G\setminus t_{3}}(\{t_{2},a_{3},b_{3}\})+\rho_{G}(A\cup\{t_{1}\})\geq 3+\rho_{G}(A\cup\{t_{1}\}).
\end{align*}
Therefore $\rho_{G}(A\cup\{t_{1}\})\leq 2$, contradicting our assumption that $G$ is $3$-rank-connected.
So we deduce that $\{t_{3},a_{2},b_{2}\}\subseteq B$. 

By Lemma~\ref{lem:intersection2}, $|A\cap\{t_{2},t_{3},a_{1},b_{1}\}|=|B\cap\{t_{2},t_{3},a_{1},b_{1}\}|=2$. So $|A\cap\{a_{1},b_{1}\}|=|B\cap\{a_{1},b_{1}\}|=1$ and we can assume that $\{a_{1},t_{2},a_{3},b_{3}\}\subseteq A$ and $\{b_{1},t_{3},a_{2},b_{2}\}\subseteq B$ by swapping $a_{1}$ and $b_{1}$ if necessary. 

If $|A|=4$, then $A$ is sequential in $G*t_{1}\setminus t_{1}$ because $\rho_{G*t_{1}\setminus t_{1}}(\{t_{2},a_{3},b_{3}\})\leq 2$ and $\{t_{2},a_{3},b_{3}\}\subseteq A$, contradicting our assumption on $A$.
Hence $|A|\geq 5$. 

If $|B|=4$, then $B$ is sequential in $G*t_{1}\setminus t_{1}$ because $\rho_{G*t_{1}\setminus t_{1}}(\{t_{3},a_{2},b_{2}\})\leq 2$ and $\{t_{3},a_{2},b_{2}\}\subseteq B$, contradicting our assumption on $B$. So $|B|\geq 5$ and $|V(G)|=|A|+|B|+1\geq 11$.

For each $k\in\{1,2,3\}$, let $P_{k}=Q_{k}\cup\{t_{k}\}=\{t_{1},t_{2},t_{3},a_{k},b_{k}\}$. 
Observe that $\rho_{G}(P_{k})\leq \rho_{G\setminus t_{k}}(Q_{k})+1\leq 3$ for each $1\leq k\leq 3$. Since $G$ is $3$-rank-connected and $|P_{1}\cap P_{3}|=3$, we have $\rho_{G}(P_{1}\cap P_{3})\geq 3$. By Lemma~\ref{lem:subeq}, 
\[
6\geq\rho_{G}(P_{1})+\rho_{G}(P_{3})\geq\rho_{G}(P_{1}\cup P_{3})+\rho_{G}(P_{1}\cap P_{3})\geq\rho_{G}(P_{1}\cup P_{3})+3,
\]
which implies that $\rho_{G}(P_{1}\cup P_{3})\leq 3$. Observe that $| V(G)-(A\cup(P_{1}\cup P_{3}))|\geq|B-\{b_{1},t_{3}\}|\geq 3$. Since $G$ is $3$-rank-connected, $\rho_{G}(A\cup(P_{1}\cup P_{3}))\geq 3$.
By Lemma~\ref{lem:subeq},
\[
3+3\geq\rho_{G}(A)+\rho_{G}(P_{1}\cup P_{3})\geq\rho_{G}(A\cap(P_{1}\cup P_{3}))+\rho_{G}(A\cup(P_{1}\cup P_{3}))\geq\rho_{G}(A\cap(P_{1}\cup P_{3}))+3.
\]
Therefore $\rho_{G}(\{a_{1},t_{2},a_{3},b_{3}\})=\rho_{G}(A\cap(P_{1}\cup P_{3}))\leq 3$. Hence by Lemma~\ref{lem:subeq}, 
\begin{align*}
3+2&\geq \rho_{G\setminus t_{3}}(\{a_{1},t_{2},a_{3},b_{3}\})+\rho_{G\setminus t_{3}}(\{t_{1},t_{2},a_{3},b_{3}\}) \\
&\geq\rho_{G\setminus t_{3}}(\{a_{1},t_{1},t_{2},a_{3},b_{3}\})+\rho_{G\setminus t_{3}}(\{t_{2},a_{3},b_{3}\})=\rho_{G\setminus t_{3}}(\{a_{1},t_{1},t_{2},a_{3},b_{3}\})+3.
\end{align*}
Hence $\rho_{G\setminus t_{3}}(\{a_{1},t_{1},t_{2},a_{3},b_{3}\})\leq 2$, contradicting our assumption that $G\setminus t_{3}$ is weakly $3$-rank-connected.
\end{proof}

An \emph{independent} set of a graph is a set of pairwise nonadjacent vertices. For sets $A$ and $B$, let $A\triangle B=(A-B)\cup(B-A)$.

\begin{lemma}
\label{lem:nonzero_quad}
Let $G$ be a $3$-rank-connected graph with $|V(G)|\geq 6$ and $x$ be a vertex of $G$ such that $G\setminus x$ is weakly $3$-rank-connected. Let $P$ be a quad of $G\setminus x$. Then there is a graph $G'$ locally equivalent to $G$ such that the following hold.
\begin{enumerate}[label=\rm(\arabic*)]
\item $G'\setminus v$ is weakly $3$-rank-connected for each vertex $v\in P\cup\{x\}$.
\item $N_{G'}(t)-P\neq\emptyset$ for each $t\in P$.
\item $P$ is a quad of $G'\setminus x$.
\end{enumerate}
\end{lemma}
\begin{proof}
Let $P=\{p,q,r,s\}$. By Lemma~\ref{lem:weak_conn}, there is a graph locally equivalent to $G$ satisfying (1) and (3). We may assume that among all graphs locally equivalent to $G$ satisfying (1) and (3), $G$ maximizes the number of edges between vertices in $P$. 

We may assume that $N_{G}(p)\subseteq\{q,r,s\}$ because otherwise (1), (2), and (3) hold for $G'=G$. Since $P$ is a quad of $G\setminus x$, we have $\rho_{G\setminus x}(P)=2$, which implies that $|V(G\setminus x)-P|\geq 2$. 
So $|V(G)|\geq 7$. Since $G$ is $3$-rank-connected, by Lemma~\ref{lem:deg3}, we have $N_{G}(p)=\{q,r,s\}$. 

Suppose that $\{q,r,s\}$ is independent in $G$. Since $G$ is $3$-rank-connected, by Lemma~\ref{lem:kconn},~$G$~is $3$-connected and so $G\setminus x\setminus p$ is connected. Let $X$ be a shortest path joining two vertices of $\{q,r,s\}$ in $G\setminus x\setminus p$. By symmetry, we may assume that $X=qv_{1}\cdots v_{m}r$ and $v_{i}\neq s$ for each $1\leq i\leq m$. Since $\{q,r,s\}$ is independent in $G$, we deduce that $m\geq 1$ and $\{v_{1},\ldots,v_{m}\}\subseteq V(G)-(P\cup\{x\})$. Then let $G'=G*v_{1}*\cdots*v_{m}$. Then $G'$ satisfies (1) and (3). Moreover, $N_{G'}(p)=\{q,r,s\}$ and $qr\in E(G')$. Hence $|E(G'[P])|>|E(G[P])|$, contradicting the choice of $G$. Therefore, $\{q,r,s\}$ is not independent in $G$.

Since $G$ is $3$-rank-connected, we have $3\leq\rho_{G}(P)\leq\rho_{G\setminus x}(P)+1=3$. Hence $\rho_{G}(P)=3$ and so $N_{G}(q)-P$, $N_{G}(r)-P$, and $N_{G}(s)-P$ are nonempty, pairwise distinct, and $(N_{G}(s)-P)\triangle (N_{G}(q)-P)\triangle(N_{G}(r)-P)\neq\emptyset$.

If $G*q\setminus q$ is weakly $3$-rank-connected, then let $G'=G*q$. Obviously, (1) and (3) hold. We have $N_{G'}(p)-P=N_{G}(q)-P=N_{G'}(q)-P\neq\emptyset$. For each vertex $v\in \{r,s\}$, 
\[
N_{G'}(v)-P=
\begin{cases}
N_{G}(v)-P\neq\emptyset & \text{if $v$ is not adjacent to $q$ in $G$,} \\
(N_{G}(q)-P)\triangle(N_{G}(v)-P)\neq\emptyset & \text{if $v$ is adjacent to $q$ in $G$,}
\end{cases}
\]
and therefore $G'$ satisfies (2).
So we can assume that none of $G*q\setminus q$, $G*r\setminus r$, and $G*s\setminus s$ is weakly $3$-rank-connected. Then by Lemma~\ref{lem:weakly}, all of $G/q$, $G/r$, and $G/s$ are weakly $3$-rank-connected.

Since $\{q,r,s\}$ is not independent in $G$, by symmetry, we may assume that $q$ and $r$ are adjacent in $G$.
Let $G'=G\wedge qr$. For each vertex $v\in P\cup\{x\}$, if $v\in\{p,s,x\}$, then $G'\setminus v=(G\setminus v)\wedge qr$ and if $v\in\{q,r\}$, then  $G'\setminus v=G/v$, which implies that (1) and (3) hold.
Then $N_{G'}(q)-P=N_{G}(r)-P$ and $N_{G'}(r)-P=N_{G}(q)-P$. Since $p\in N_{G}(q)\cap N_{G}(r)$ and $N_{G}(q)-P \neq N_{G}(r)-P$, we have $N_{G'}(p)-P=(N_{G}(q)-P)\triangle (N_{G}(r)-P)\neq\emptyset$. Furthermore, 
\[
N_{G'}(s)-P=
\begin{cases}
N_{G}(s)-P\neq\emptyset & \text{if $s\notin N_{G}(q)\cup N_{G}(r)$,} \\
(N_{G}(s)-P)\triangle (N_{G}(q)-P)\neq\emptyset & \text{if $s\in N_{G}(r)-N_{G}(q)$,} \\
(N_{G}(s)-P)\triangle (N_{G}(r)-P)\neq\emptyset & \text{if $s\in N_{G}(q)-N_{G}(r)$,} \\
(N_{G}(s)-P)\triangle (N_{G}(q)-P)\triangle (N_{G}(r)-P)\neq\emptyset & \text{if $s\in N_{G}(q)\cap N_{G}(r)$.} \\
\end{cases}
\]
Hence, (2) holds. 
\end{proof}

\begin{lemma}
\label{lem:find_cross}
Let $G$ be a $3$-rank-connected graph with $|V(G)|\geq 6$ and $x$ be a vertex of $G$. Let $P$ be a quad of $G\setminus x$ and $t$ be a vertex in $P$. If $G\setminus t$ is weakly $3$-rank-connected, then one of the following holds.
\begin{enumerate}[label=\rm(Q\arabic*)]
\item\label{item:Q1} $G\setminus t$ is sequentially $3$-rank-connected.
\item\label{item:Q2} There is a subset $X$ of $V(G\setminus t)$ such that $\rho_{G\setminus t}(X)\leq 2$, $X\cap P\neq\emptyset$, $(V(G\setminus t)-X)\cap P\neq\emptyset$, and neither $X$ nor $V(G\setminus t)-X$ is sequential in $G\setminus t$.
\item\label{item:Q3} $\rho_{G\setminus t}(P-\{t\})=2$ and $G\setminus t$ has a quad $Y$ containing $x$ such that $Y\cap P=\emptyset$.
\end{enumerate}
\end{lemma}
\begin{proof}
Suppose that $G\setminus t$ is not sequentially $3$-rank-connected. Then there is a subset $X$ of $V(G\setminus t)$ such that $\rho_{G\setminus t}(X)\leq 2$ and neither $X$ nor $V(G\setminus t)-X$ is sequential in $G\setminus t$. Let $Y=V(G\setminus t)-X$ and $(Z_1,Z_2)=(X-\{x\},Y-\{x\})$. Since both $X$ and $Y$ are non-sequential in $G\setminus t$, we have $|X|, |Y|\geq 4$ and so $|Z_{1}|, |Z_{2}|\geq 3$. If $X\cap P\neq\emptyset$ and $Y\cap P\neq\emptyset$, then \ref{item:Q2} holds. So by symmetry, we may assume that $P-\{t\}\subseteq X$. Then $P-\{t\}\subseteq Z_{1}$. Since $P$ is a quad of $G\setminus x$, we know that $\rho_{G\setminus x}(P)=2=\rho_{G\setminus x}(P-\{t\})-1\leq\rho_{G\setminus x\setminus t}(P-\{t\})$. Then by~Lemma~\ref{lem:cor_s1}, $\rho_{G\setminus x}(Z_1\cup\{t\})=\rho_{G\setminus x\setminus t}(Z_1)$.

By Lemma~\ref{lem:3toprime}, $G\setminus x$ is prime and so $2\leq\rho_{G\setminus x}(Z_1\cup\{t\})=\rho_{G\setminus x\setminus t}(Z_1)\leq\rho_{G\setminus t}(X)\leq 2$, which implies that 
\[
\rho_{G\setminus x}(Z_1\cup\{t\})=\rho_{G\setminus x\setminus t}(Z_1)=2.
\]
Since $G$ is $3$-rank-connected and $|V(G)-(Z_{1}\cup\{x,t\})|\geq|Z_{2}|\geq 3$, we have $\rho_{G}(Z_{1}\cup\{x,t\})\geq 3$. So by Lemma~\ref{item:ab3} of Lemma~\ref{lem:sub_eq_AB}, 
\[
2+\rho_{G\setminus t}(Z_{1}\cup\{x\})\geq \rho_{G\setminus x}(Z_{1}\cup\{t\})+\rho_{G\setminus t}(Z_{1}\cup\{x\})\geq \rho_{G}(Z_{1}\cup\{x,t\})+\rho_{G\setminus x\setminus t}(Z_{1})\geq 3+2.
\]
Hence $\rho_{G\setminus t}(Z_{1}\cup\{x\})>2$ and $x\in Y$. So $(Z_{1},Z_{2})=(X,Y-\{x\})$ and 
$\rho_{G\setminus x}(X\cup\{t\})=\rho_{G\setminus x}(Z_1\cup\{t\})=2$. Since $t\in P$ and $x\notin Z_{1}$, by \ref{item:ab2} of Lemma~\ref{lem:sub_eq_AB}, 
\[
2+2\geq\rho_{G\setminus x}(P)+\rho_{G\setminus t}(Z_{1})\geq\rho_{G\setminus x}(Z_{1}\cup\{t\})+\rho_{G\setminus t}(P-\{t\})\geq 2+\rho_{G\setminus t}(P-\{t\}).
\]
Therefore, $\rho_{G\setminus t}(P-\{t\})=2$ because $G\setminus t$ is prime. Since $X$ is non-sequential in $G\setminus t$ and $\rho_{G\setminus t}(P-\{t\})\leq 2$, we have $|X|\geq 5$. Hence $|Y|=4$ because $G\setminus t$ is weakly $3$-rank-connected. Since $Y$ is non-sequential in $G\setminus t$, by Lemma~\ref{lem:quad}, $Y=Z_{2}\cup\{x\}$ is a quad of $G\setminus t$. Hence~\ref{item:Q3} holds.
\end{proof}

\begin{lemma}
\label{lem:less_than4}
Let $G$ be a $3$-rank-connected graph such that $|V(G)|\geq 12$ and $x$ be a vertex of $G$. Let $P$ be a quad of $G\setminus x$ and $t$ be a vertex of $P$. Let $(X,Y)$ be a partition of $V(G)-\{t\}$ such that $\rho_{G\setminus t}(X)\leq 2$ and neither $X$ nor $Y$ is sequential in $G\setminus t$. If $G\setminus x$ and $G\setminus t$ are weakly $3$-rank-connected and $|X\cap P|=1$, then the following hold.
\begin{enumerate}[label=\rm(K\arabic*)]
\item\label{item:k1} $\rho_{G\setminus x\setminus t}(X-\{x\})=\rho_{G\setminus t}(X)=2$.
\item\label{item:k2} $X$ is a quad of $G\setminus t$ containing $x$.
\end{enumerate}
\end{lemma}
\begin{proof}
Since neither $X$ nor $Y$ is sequential in $G\setminus t$, we have $|X|,|Y|\geq 4$ and so $|X-\{x\}|,|Y-\{x\}|\geq 3$. Clearly, $\rho_{G\setminus x\setminus t}(X-\{x\})\leq\rho_{G\setminus t}(X)\leq 2$. 
Let $q$ be the element of $X \cap P$ and $r$, $s$ be the elements of $Y \cap P$. 
Let $C=X-\{q,x\}$ and $D=Y-\{r,s,x\}$. Then we have $|D|\geq 1$ because $|Y|\geq 4$. 

Let us show that $\rho_{G\setminus x\setminus t}(C)\leq 2$. Since $P$ is a quad of $G\setminus x$, by (ii) of Lemma~\ref{lem:delrank}, $\rho_{G\setminus x\setminus t}(P-\{t\})\leq\rho_{G\setminus x}(P)=2=\rho_{G\setminus x}(P-\{q\})-1\leq\rho_{G\setminus x\setminus t}(\{r,s\})$. Hence, by Lemma~\ref{lem:subeq_minus}, 
\begin{align*}
\rho_{G\setminus x\setminus t}(X-\{x\})+2&\geq\rho_{G\setminus x\setminus t}(X-\{x\})+\rho_{G\setminus x\setminus t}(P-\{t\}) \\
&\geq \rho_{G\setminus x\setminus t}(C)+\rho_{G\setminus x\setminus t}(\{r,s\})\geq \rho_{G\setminus x\setminus t}(C)+2
\end{align*}
and therefore $\rho_{G\setminus x\setminus t}(C)\leq\rho_{G\setminus x\setminus t}(X-\{x\})\leq 2$. Since $P$ is a quad of $G\setminus x$, by (i) of Lemma~\ref{lem:delrank}, $\rho_{G\setminus x}(P)=2=\rho_{G\setminus x}(P-\{t\})-1\leq\rho_{G\setminus x\setminus t}(P-\{t\})$.
By Lemma~\ref{lem:subtool_minus}, 
\[
2+\rho_{G\setminus x\setminus t}(C)\geq \rho_{G\setminus x}((P-\{t\})\cup\{t\})+\rho_{G\setminus x\setminus t}(C)\geq \rho_{G\setminus x}(C)+\rho_{G\setminus x\setminus t}(P-\{t\})\geq \rho_{G\setminus x}(C)+2,
\]
which implies that $\rho_{G\setminus x}(C)\leq\rho_{G\setminus x\setminus t}(C)\leq 2$. Hence $\rho_{G\setminus x}(C)=\rho_{G\setminus x\setminus t}(C)=\rho_{G\setminus x\setminus t}(X-\{x\})=\rho_{G\setminus t}(X)=2$ because $G\setminus x$ is prime and $|V(G\setminus x)-C|\geq 2$. Hence~\ref{item:k1} holds. 

Since $G\setminus x$ is weakly $3$-rank-connected and $|V(G\setminus x)-C|\geq |P|+|D|\geq 5$, we deduce that $|C|\leq 4$ and $|X|\leq 6$. So $|Y|\geq 11-|X|\geq 5$.

Suppose that $x\notin X$. Then $X=X-\{x\}$ and $\rho_{G\setminus x\setminus t}(X)=\rho_{G\setminus t}(X)=2$. Since $C\subseteq X$, by Lemma~\ref{lem:cor_s2}, we have $\rho_{G\setminus t}(C)=\rho_{G\setminus x\setminus t}(C)=2$. By~\ref{item:ab1} of Lemma~\ref{lem:sub_eq_AB},
\[
\rho_{G\setminus x}(C)+\rho_{G\setminus t}(C)\geq\rho_{G}(C)+\rho_{G\setminus x\setminus t}(C),
\]
which implies that $\rho_{G}(C)\leq 2$. So $|C|\leq 2$ because $G$ is $3$-rank-connected. Then $|X|=|C\cup\{q\}|\leq 3$, contradicting our assumption on $X$. Hence $x\in X$.

Since $G\setminus t$ is weakly $3$-rank-connected, $\rho_{G\setminus t}(X)=2$, and $|Y|\geq 5$, we have 
$|X|= 4$. Therefore, by Lemma~\ref{lem:quad}, $X$ is a quad of $G\setminus t$ and~\ref{item:k2} holds.
\end{proof}

\begin{lemma}
\label{lem:crossing_two_quad}
Let $G$ be a $3$-rank-connected graph with $|V(G)|\geq 12$ and no sequentially $3$-rank-connected vertex-minor on $|V(G)|-1$ vertices. Let $x$ be a vertex of $G$ such that $G\setminus x$ is weakly $3$-rank-connected and $P$ be a quad of $G\setminus x$. Then there is a graph $G'$ locally equivalent to $G$ such that the following hold.
\begin{enumerate}[label=\rm(\arabic*)]
\item $G'\setminus v$ is weakly $3$-rank-connected for each vertex $v$ of $P\cup\{x\}$.
\item $P$ is a quad of $G'\setminus x$.
\item There exist a $2$-element subset $S$ of $P$ and a quad $X_{u}$ of $G'\setminus u$ for each $u$ in $S$ such that $x\in X_{u}$, $|X_{u}\cap P|=1$, and $V(G'\setminus u)-X_{u}$ is not sequential in $G'\setminus u$.
\end{enumerate}
\end{lemma}
\begin{proof}
By Lemma~\ref{lem:3toprime}, $G\setminus v$ is prime for each vertex $v$ of $G$.
By Lemma~\ref{lem:nonzero_quad}, we can assume that $G\setminus v$ is weakly $3$-rank-connected for each vertex $v$ of $P\cup\{x\}$, the set $P$ is a quad of $G\setminus x$, and $N_{G}(t)-P$ is nonempty for each $t\in P$. 

By Lemma~\ref{lem:find_cross}, each vertex $t$ in $P$ satisfies \ref{item:Q2} or \ref{item:Q3}.
Suppose that at most $1$ vertex of $P$ satisfies~\ref{item:Q2}. Then by Lemma~\ref{lem:find_cross}, there exist $3$ vertices $q$, $r$, $s$ of $P$ such that $\rho_{G\setminus q}(P-\{q\})=2$, $\rho_{G\setminus r}(P-\{r\})=2$, and $\rho_{G\setminus s}(P-\{s\})=2$. Since $P$ is a quad of $G\setminus x$, by (i) of Lemma~\ref{lem:delrank}, we have $\rho_{G}(P)\leq\rho_{G\setminus x}(P)+1\leq 3$. Since $G$ is $3$-rank-connected, $3\leq\rho_{G}(P)$ and therefore, $\rho_{G}(P)=3$.
By \ref{item:ab3} of Lemma~\ref{lem:sub_eq_AB}, 
\begin{align*}
2+2&=\rho_{G\setminus q}(P-\{q\})+\rho_{G\setminus r}(P-\{r\}) \\
&\geq\rho_{G}(P)+\rho_{G\setminus q\setminus r}(P-\{q,r\})=3+\rho_{G\setminus q\setminus r}(P-\{q,r\}).
\end{align*}
Therefore, $\rho_{G\setminus q\setminus r}(P-\{q,r\})\leq 1$ and by symmetry, $\rho_{G\setminus q\setminus s}(P-\{q,s\})\leq 1$ and $\rho_{G\setminus r\setminus s}(P-\{r,s\})\leq 1$. Let $p$ be the element of $P-\{q,r,s\}$. Since $N_{G}(t)-P\neq\emptyset$ for each $t\in P$, we have $N_{G}(p)-P=N_{G}(q)-P=N_{G}(r)-P=N_{G}(s)-P$ and therefore $\rho_{G}(P)=1$, contradicting our assumption. 

Therefore, there exist a subset $S=\{p,q\}$ of $P$ and a subset $X_u$ of $V(G\setminus u)$ for each $u\in S$ such that $\rho_{G\setminus u}(X_u)\leq 2$, both $X_u\cap P$ and $(V(G\setminus u)-X_u)\cap P$ are nonempty, and neither $X_u$ nor $V(G\setminus u)-X_u$ is sequential in $G\setminus u$. 

Let $Y_p=V(G\setminus p)-X_{p}$ and $Y_q=V(G\setminus q)-X_{q}$. By symmetry, we may assume that $|X_{p}\cap P|=1$ and $|X_{q}\cap P|=1$. Then by~\ref{item:k2} of Lemma~\ref{lem:less_than4}, $X_{p}$ is a quad of $G\setminus p$, $X_q$ is a quad of $G\setminus q$, and $x\in X_{p}\cap X_{q}$. 
\end{proof}

\begin{lemma}
\label{lem:quad_intersection}
Let $G$ be a $3$-rank-connected graph with $|V(G)|\geq 12$ and $x$, $y$ be distinct vertices of~$G$ such that
both $G\setminus x$ and $G\setminus y$ are weakly $3$-rank-connected. Let $A$ be a quad of $G\setminus x$ and $B$ be a quad of $G\setminus y$. Then $|A\cap B|\leq 2$. 
\end{lemma}
\begin{proof}
Suppose that $|A\cap B|\geq 3$. First let us consider the case when $y\notin A$ and $x\notin B$. Since $G$ is $3$-rank-connected and $|V(G)-(A\cap B)|\geq 3$, we have $\rho_{G}(A\cap B)\geq 3$. So by~\ref{item:ab1} of Lemma~\ref{lem:sub_eq_AB}, 
\[
2+2\geq \rho_{G\setminus x}(A)+\rho_{G\setminus y}(B)\geq \rho_{G}(A\cap B)+\rho_{G\setminus x\setminus y}(A\cup B)\geq 3+\rho_{G\setminus x\setminus y}(A\cup B).
\]
Hence $\rho_{G\setminus x\setminus y}(A\cup B)\leq 1$. Then by (ii) of Lemma~\ref{lem:delrank}, we have $\rho_{G\setminus x}(A\cup B\cup\{y\})\leq 2$. Since $G\setminus x$ is weakly $3$-rank-connected and $|A\cup B\cup\{y\}|\in\{5,6\}$, we deduce that $|V(G\setminus x)-(A\cup B\cup\{y\})|\leq 4$ and so $|V(G)|\leq 11$, contradicting our assumption.

Now we consider the case when either
\begin{itemize}
\item $y\in A$ and $x\notin B$, or
\item $y\notin A$ and $x\in B$.
\end{itemize}
By symmetry, we may assume that $y\in A$ and $x\notin B$.
Then $|A\cap B|=3$ because $x\notin B$. Since $G\setminus x$ is weakly $3$-rank-connected, $|A\cup B|=5$, and $|V(G\setminus x)-(A\cup B)|\geq 6$, we have $\rho_{G\setminus x}(A\cup B)\geq 3$. By \ref{item:ab2} of Lemma~\ref{lem:sub_eq_AB}, 
\[
2+2\geq \rho_{G\setminus x}(A)+\rho_{G\setminus y}(B)\geq\rho_{G\setminus x}(A\cup B)+\rho_{G\setminus y}(A\cap B)\geq 3+\rho_{G\setminus y}(A\cap B).
\]
Hence $\rho_{G\setminus y}(A\cap B)\leq 1$, contradicting the fact that $G\setminus y$ is prime.

Now it remains to consider the case when $y\in A$ and $x\in B$. Since $x\notin A$ and $y\notin B$, we have $|A\cap B|=3$. Since $G$ is $3$-rank-connected and $|V(G)-(A\cup B)|\geq 7$, we have $\rho_{G}(A\cup B)\geq 3$. By~\ref{item:ab3} of Lemma~\ref{lem:sub_eq_AB},
\[
2+2\geq\rho_{G\setminus x}(A)+\rho_{G\setminus y}(B)\geq \rho_{G}(A\cup B)+\rho_{G\setminus x\setminus y}(A\cap B)\geq 3+\rho_{G\setminus x\setminus y}(A\cap B).
\]
So $\rho_{G\setminus x\setminus y}(A\cap B)\leq 1$ and $\rho_{G\setminus x}(A\cap B)\leq 2$, contradicting the assumption that $A$ is a quad of $G\setminus x$.
\end{proof}

\begin{lemma}
\label{lem:contain_quad}
Let $G$ be a $3$-rank-connected graph with $|V(G)|\geq 12$ and $x$ be a vertex of $G$. Let $P$ be a quad of $G\setminus x$ and $y$ be a vertex of $P$. Let $Q$ be a quad of $G\setminus y$. If $G\setminus x$ is weakly $3$-rank-connected and $|P\cap Q|=2$, then $x\in Q$.
\end{lemma}
\begin{proof}
Suppose that $x\notin Q$. Since $G$ is $3$-rank-connected, by Lemma~\ref{lem:3toprime}, $G\setminus y$ is prime. Therefore, $\rho_{G\setminus y}(P\cap Q)=2$ because $|P\cap Q|=2$. 
Since $y\in P$ and $x\notin Q$, by \ref{item:ab2} of Lemma~\ref{lem:sub_eq_AB}, 
\[
2+2\geq\rho_{G\setminus x}(P)+\rho_{G\setminus y}(Q)\geq\rho_{G\setminus x}(P\cup Q)+\rho_{G\setminus y}(P\cap Q)\geq\rho_{G\setminus x}(P\cup Q)+2.
\]
Hence $\rho_{G\setminus x}(P\cup Q)\leq 2$. Since $G\setminus x$ is weakly $3$-rank-connected and $|P\cup Q|=6$, we have $|V(G\setminus x)-(P\cup Q)|\leq 4$ and so $|V(G)|\leq 11$, contradicting our assumption.
\end{proof}

\begin{lemma}
\label{lem:disjoint_2set}
Let $G$ be a $3$-rank-connected graph with $|V(G)|\geq 13$ and $x$ be a vertex of $G$. Let $P$ be a quad of $G\setminus x$ and $p$, $q$ be distinct vertices of $P$. For each $u\in\{p,q\}$, let $A_{u}$ be a quad of $G\setminus u$ such that $x\in A_{u}$, $|A_{u}\cap P|=1$, and $V(G\setminus u)-A_{u}$ is not sequential in $G\setminus u$. If $G\setminus x$, $G\setminus p$, and $G\setminus q$ are weakly $3$-rank-connected, then $A_{p}\cap A_{q}\subseteq P\cup\{x\}$.
\end{lemma}
\begin{proof}
For each $u\in\{p,q\}$, let $B_{u}=A_{u}-(P\cup\{x\})$. Then $|B_{u}|=2$ and $|A_{u}\cup P|=7$ for each $u\in\{p,q\}$. Let $t$ be the unique element of $A_{p}\cap P$.

Now we claim that $\rho_{G}(A_{p}\cup P)=3$. 
By Lemma~\ref{lem:deg3}, $N_{G\setminus x\setminus p}(t)\neq\emptyset$ and so $\rho_{G\setminus x\setminus p}(\{t\})=1$. Since $P$ is a quad of $G\setminus x$, we have $\rho_{G\setminus x\setminus p}(P-\{p\})\leq\rho_{G\setminus x}(P)=2$. By~\ref{item:k1} of Lemma~\ref{lem:less_than4}, $\rho_{G\setminus x\setminus p}(A_{p}-\{x\})=\rho_{G\setminus p}(A_{p})=2$. By Lemma~\ref{lem:cor_s1}, 
$\rho_{G\setminus p}(A_{p}\cup (P-\{p\}))=\rho_{G\setminus x\setminus p}((A_{p}-\{x\})\cup (P-\{p\}))$.

By Lemma~\ref{lem:subeq}, 
\begin{align*}
2+2&\geq\rho_{G\setminus x\setminus p}(A_{p}-\{x\})+\rho_{G\setminus x\setminus p}(P-\{p\}) \\ &\geq\rho_{G\setminus x\setminus p}((A_{p}-\{x\})\cup (P-\{p\}))+\rho_{G\setminus x\setminus p}(\{t\})\geq\rho_{G\setminus x\setminus p}((A_{p}-\{x\})\cup (P-\{p\}))+1.
\end{align*}
Hence $\rho_{G\setminus x\setminus p}((A_{p}-\{x\})\cup (P-\{p\}))\leq 3$. 
 
Since $P$ is a quad of $G\setminus x$, we have $\rho_{G\setminus x}(P)=2=\rho_{G\setminus x}(P-\{p\})-1\leq\rho_{G\setminus x\setminus p}(P-\{p\})$.
So by Lemma~\ref{lem:cor_s1}, 
$\rho_{G\setminus x}((A_{p}-\{x\})\cup P)=\rho_{G\setminus x\setminus p}((A_{p}-\{x\})\cup(P-\{p\}))$.
By~\ref{item:ab3} of Lemma~\ref{lem:sub_eq_AB},
\[
\rho_{G\setminus x}((A_{p}-\{x\})\cup P)+\rho_{G\setminus p}(A_{p}\cup (P-\{p\}))\geq\rho_{G}(A_{p}\cup P)+\rho_{G\setminus x\setminus p}((A_{p}-\{x\})\cup (P-\{p\})).
\]
It follows that $\rho_{G}(A_{p}\cup P)=\rho_{G\setminus x\setminus p}((A_{p}-\{x\})\cup (P-\{p\}))\leq 3$. 
Since $G$ is $3$-rank-connected and $|A_{p}\cup P|, |V(G)-(A_{p}\cup P)|\geq 3$, we have $\rho_{G}(A_{p}\cup P)=3$. 

By Lemma~\ref{lem:quad_intersection}, $|A_{p}\cap A_{q}|\leq 2$. Since $x\in A_{p}\cap A_{q}$, we have $|B_{p}\cap B_{q}|\leq 1$.

Suppose that $|B_{p}\cap B_{q}|=1$. Then  $|A_{q}\cap(A_{p}\cup P)|=|A_{q}|-|A_{q}-(A_{p}\cup P)|=|A_{q}|-|B_{q}-B_{p}|=|A_{q}|-(|B_{q}|-|B_{p}\cap B_{q}|)=3$.
So $\rho_{G\setminus q}(A_{q}\cap(A_{p}\cup P))=3$ because $A_{q}$ is a quad of $G\setminus q$. 
Since $\rho_{G\setminus q}(A_{q})=2$ and $\rho_{G\setminus q}((A_{p}\cup P)-\{q\})\leq\rho_{G}(A_{p}\cup P)=3$, by Lemma~\ref{lem:subeq_minus}, 
\begin{align*}
5\geq \rho_{G\setminus q}(A_{q})+\rho_{G\setminus q}((A_{p}\cup P)-\{q\})&\geq \rho_{G\setminus q}((A_{q}\cup (A_{p}\cup P))-\{q\})+\rho_{G\setminus q}(A_{q}\cap (A_{p}\cup P)) \\
&=\rho_{G\setminus q}((A_{q}\cup (A_{p}\cup P))-\{q\})+3.
\end{align*}
Hence $\rho_{G\setminus q}((A_{q}\cup(A_{p}\cup P))-\{q\})\leq 2$. Since $G\setminus q$ is weakly $3$-rank-connected and $|(A_{q}\cup(A_{p}\cup P))-\{q\}|=|A_{q}|+|A_{p}\cup P|-|A_{q}\cap(A_{p}\cup P)|-1=7$, we deduce that $|V(G\setminus q)-((A_{q}\cup(A_{p}\cup P))-\{q\})|\leq 4$. 
Therefore, $|V(G)|\leq 12$, contradicting our assumption. Therefore, $B_{p}\cap B_{q}=\emptyset$ and so $A_{p}\cap A_{q}\subseteq P\cup\{x\}$.
\end{proof}

\begin{lemma}
\label{lem:different_quad}
Let $G$ be a $3$-rank-connected graph with $|V(G)|\geq 6$ and $a$, $b$ be distinct vertices of~$G$. Let $A$ be a quad of $G\setminus a$ and $B$ be a quad of $G\setminus b$. If $|A\cap B|=1$, then $b\in A$ and $a\in B$.
\end{lemma}
\begin{proof}
Suppose not. Then by symmetry, we may assume that $b\notin A$. Since $B$ is a quad of $G\setminus b$, we know that $\rho_{G\setminus b}(B)<\rho_{G\setminus b}(B-A)$. Then by Lemma~\ref{lem:subeq_minus}, 
\[
\rho_{G\setminus b}(B)+\rho_{G\setminus b}(A)\geq\rho_{G\setminus b}(A-B)+\rho_{G\setminus b}(B-A)
\]
and therefore $\rho_{G\setminus b}(A-B)<\rho_{G\setminus b}(A)$. Since $A$ is a quad of $G\setminus a$, we have that $\rho_{G}(A)\leq\rho_{G\setminus a}(A)+1\leq 3$. By Lemma~\ref{lem:3toprime}, $G\setminus b$ is prime and so
\[
2\leq\rho_{G\setminus b}(A-B)<\rho_{G\setminus b}(A)\leq\rho_{G}(A)\leq 3,
\]
which implies that $\rho_{G\setminus b}(A-B)=2$ and $\rho_{G\setminus b}(A)=3$. Since $2=\rho_{G\setminus b}(A)-1\leq\rho_{G\setminus a\setminus b}(A)\leq\rho_{G\setminus a}(A)=2$, we have $\rho_{G\setminus a\setminus b}(A)=2$. Since $a\notin A-B$ and $b\notin A$, by \ref{item:ab1} of Lemma~\ref{lem:sub_eq_AB}, 
\[
2+2=\rho_{G\setminus a}(A)+\rho_{G\setminus b}(A-B)\geq\rho_{G}(A-B)+\rho_{G\setminus a\setminus b}(A)=\rho_{G}(A-B)+2.
\]
Hence $\rho_{G}(A-B)\leq 2$, contradicting the condition that $G$ is $3$-rank-connected.
\end{proof}

\begin{proposition}
\label{prop:3rank_conn}
Let $G$ be a $3$-rank-connected graph such that $|V(G)|\geq 13$. Then there exists a sequentially $3$-rank-connected vertex-minor $H$ of $G$ such that $|V(H)|=|V(G)|-1$.
\end{proposition}

\begin{proof}
Suppose that no vertex-minor of $G$ on $|V(G)|-1$ vertices is sequentially $3$-rank-connected. Let $x$ be a vertex of $G$. By Lemma~\ref{lem:weak_conn}, we can assume that $G\setminus x$ is weakly $3$-rank-connected.
By Lemma~\ref{lem:3toprime}, $G\setminus x$ is prime. Since $G\setminus x$ is not sequentially $3$-rank-connected, there exists a subset $P$ of $V(G\setminus x)$ such that $\rho_{G\setminus x}(P)\leq 2$ and neither $P$ nor $V(G\setminus x)-P$ is sequential in $G\setminus x$. Since $G\setminus x$ is weakly $3$-rank-connected, we may assume that $|P|=4$. Since $|V(G\setminus x)-P|\geq 4$ and $G\setminus x$ is prime, $\rho_{G\setminus x}(P)=2$. So by Lemma~\ref{lem:quad}, $P$ is a quad of $G\setminus x$. Then by Lemma~\ref{lem:crossing_two_quad}, we can assume the following.
\begin{enumerate}[label=\rm(\arabic*)]
\item $G\setminus v$ is weakly $3$-rank-connected for each vertex $v$ of $P\cup\{x\}$.
\item $P$ is a quad of $G\setminus x$.
\item There exist a $2$-element subset $S$ of $P$ and a quad $X_{u}$ of $G\setminus u$ for each $u$ in $S$ such that $x\in X_{u}$, $|X_{u}\cap P|=1$, and $V(G\setminus u)-X_{u}$ is not sequential in $G\setminus u$.
\end{enumerate}
Let $p$ and $q$ be distinct vertices of $S$. By Lemma~\ref{lem:disjoint_2set}, $x\in X_{p}\cap X_{q}\subseteq P\cup\{x\}$. 
By Lemma~\ref{lem:quad_intersection}, $|X_{p}\cap X_{q}|\leq 2$.

If $|X_{p}\cap X_{q}|=1$, then, by Lemma~\ref{lem:different_quad}, $q\in X_p$ and $p\in X_q$.
Then, since $X_{p}\cap X_{q}=\{x\}$, $X_{p}\cap P=\{q\}$, and $X_{q}\cap P=\{p\}$, by Proposition~\ref{prop:key}, $G*x\setminus x$ or $G/x$ is sequentially $3$-rank-connected, contradicting the assumption.

So $|X_{p}\cap X_{q}|=2$. Let $r\in X_{p}\cap X_{q}-\{x\}$. Since $r$ does not satisfy~\ref{item:Q1}, by Lemma~\ref{lem:find_cross}, \ref{item:Q2} or~\ref{item:Q3} holds for $r$. 

If \ref{item:Q2} holds, there is a subset $R$ of $V(G\setminus r)$ such that $\rho_{G\setminus r}(R)\leq 2$, $R\cap P\neq\emptyset$,
$(V(G\setminus r)-R)\cap P\neq\emptyset$, and neither $R$ nor $V(G\setminus r)-R$ is sequential in $G\setminus r$. By symmetry, we may assume that $|P\cap R|=1$ by replacing $R$ by $V(G\setminus r)-R$. Then by~\ref{item:k2} of Lemma~\ref{lem:less_than4}, $R$ is a quad of $G\setminus r$ containing $x$. By Lemma~\ref{lem:quad_intersection}, $|R\cap X_{p}|, |R\cap X_{q}|\leq 2$. 

Suppose that $|R\cap X_{p}|=2$ and $|R\cap X_{q}|=2$. Then by applying Lemma~\ref{lem:contain_quad} twice, we deduce that $R$ contains both $p$ and $q$, contradicting our assumption that $|P\cap R|=1$.
So by symmetry, we can assume that $|R\cap X_{p}|=1$. Then by Lemma~\ref{lem:different_quad}, $p\in R$. Since $R\cap X_{p}=\{x\}$, $P\cap R=\{p\}$, and $X_{p}\cap P=\{r\}$, by Lemma~\ref{prop:key}, we deduce that $G*x\setminus x$ or $G/x$ is sequentially $3$-rank-connected, contradicting our assumption.

If~\ref{item:Q3} holds, then there is a quad of $R$ of $G\setminus r$ containing $x$ such that $R\cap P=\emptyset$. 
By Lemma~\ref{lem:quad_intersection}, $|R\cap X_{p}|\leq 2$. Since $p\notin R$, by Lemma~\ref{lem:contain_quad}, $|R\cap X_{p}|=1$. Then Lemma~\ref{lem:different_quad} implies that $p\in R$, contradicting the assumption. 
\end{proof}

\section{Treating internally $3$-rank-connected graphs}
\label{sec:internal}
In this section, we prove Theorem~\ref{thm:main} for internally $3$-rank-connected graphs.

A graph $G$ is \emph{internally $3$-rank-connected} if $G$ is prime and for each subset $X$ of $V(G)$, either $|X|\leq 3$ or $|V(G)-X|\leq 3$ whenever $\rho_{G}(X)\leq 2$.
A $3$-element set $T$ of vertices of a graph $G$ is a \emph{triplet} of~$G$ if $\rho_{G}(T)=2$ and $\rho_{G\setminus x}(T-x)=2$ for each $x\in T$. 

Here is a rough overview of our approach in this section. If $G$ is an  internally $3$-rank-connected counterexample of Theorem~\ref{thm:main} and $|V(G)|\geq 13$, then by pivoting, we may assume that $G$ has a triplet $T=\{a,b,c\}$. Next we find a partition $(A_b,A_c)$ of $V(G\setminus a)$, a partition $(B_a,B_c)$ of $V(G\setminus b)$, and a partition $(C_a,C_b)$ of $V(G\setminus c)$ satisfying the following conditions:
\begin{enumerate}
\item[(1)] $b\in A_b$, $c\in A_c$, and neither $A_b$ nor $A_c$ is sequential in $G\setminus a$. 
\item[(2)] $a\in B_a$, $c\in B_c$, and neither $B_a$ nor $B_c$ is sequential in $G\setminus b$. 
\item[(3)] $a\in C_a$, $b\in C_b$, and neither $C_a$ nor $C_b$ is sequential in $G\setminus c$. 
\end{enumerate}
We then prove that all of $A_b, A_c, B_a, B_c, C_a, C_b$ must be small, contradicting the assumption that $|V(G)|\geq 13$.

The following lemma shows that if a graph is internally $3$-rank-connected but not $3$-rank-connected, then we can apply pivoting to obtain a graph with a triplet. 

\begin{lemma}[Oum~{\cite[Lemma 5.1]{Oum2020}}]
\label{lem:triplet}
Let $G$ be a prime graph and $A$ be a $3$-element subset of $V(G)$ such that $\rho_{G}(A)=2$. Then there is a graph $G'$ pivot-equivalent to $G$ such that $A$ is a triplet of~$G'$.
\end{lemma}

\begin{lemma}[Oum~{\cite[Lemma 5.2]{Oum2020}}]
\label{lem:prime}
Let $G$ be an internally $3$-rank-connected graph and $T=\{a,b,c\}$ be a triplet of $G$. Then $G\setminus a$, $G\setminus b$, and $G\setminus c$ are prime.
\end{lemma}

\begin{lemma}
\label{lem:bc}
Let $T$ be a triplet of an internally $3$-rank-connected graph $G$ and $a\in T$. Let $(X,Y)$ be a partition of $V(G)-\{a\}$ such that $\rho_{G\setminus a}(X)\leq 2$ and neither $X$ nor $Y$ is sequential in $G\setminus a$. Then there exist $b\in X\cap T$ and $c\in Y\cap T$ such that $\rho_{G\setminus b}(X-\{b\})=\rho_{G\setminus c}(Y-\{c\})=3$.
\end{lemma}
\begin{proof}
Since neither $X$ nor $Y$ is sequential in $G\setminus a$, $|X|\geq 4$ and $|Y|\geq 4$. So $\rho_{G\setminus a}(X)=2$ because $G\setminus a$ is prime by Lemma~\ref{lem:prime}. Since $T$ is a triplet of $G$, we have $\rho_{G\setminus a}(T-\{a\})=\rho_{G}(T)$. If $T\subseteq X\cup\{a\}$, then by Lemma~\ref{lem:cor_s1}, 
$\rho_{G}(X\cup\{a\})=\rho_{G\setminus a}(X)=2$, contradicting the assumption that $G$ is internally $3$-rank-connected. 

Hence $T-\{a\}\nsubseteq X$ and similarly $T-\{a\}\nsubseteq Y$. Therefore, there exist $b\in X\cap T$ and $c\in Y\cap T$. Then $T=\{a,b,c\}$.

By (i) of Lemma~\ref{lem:delrank}, $\rho_{G}(X)\leq\rho_{G\setminus a}(X)+1\leq 3$. So by (ii) of Lemma~\ref{lem:delrank}, we have $\rho_{G\setminus b}(X-\{b\})\leq 3$ and similarly, $\rho_{G\setminus c}(Y-\{c\})\leq 3$.

Suppose that $\rho_{G\setminus c}(Y-\{c\})<3$. Since $T$ is a triplet of $G$,  by Lemma~\ref{lem:rank_subeq},
\begin{align*}
\rho_{G}(\{a,b\},Y-\{c\})+2&=\rho_{G}(\{a,b\},Y-\{c\})+\rho_{G}(\{a,b,c\},V(G)-\{a,b,c\}) \\
&\geq\rho_{G}(\{a,b,c\},Y-\{c\})+\rho_{G}(\{a,b\},V(G)-\{a,b,c\}) \\
&=\rho_{G}(\{a,b,c\},Y-\{c\})+2,
\end{align*}
and therefore $\rho_{G}(\{a,b,c\},Y-\{c\})\leq\rho_{G}(\{a,b\},Y-\{c\})$. Then by Lemma~\ref{lem:rank_subeq}, we have
\[
\rho_{G}(X\cup\{a\},Y-\{c\})+\rho_{G}(\{a,b,c\},Y-\{c\})\geq\rho_{G}(X\cup\{a,c\},Y-\{c\})+\rho_{G}(\{a,b\},Y-\{c\}).
\]
Hence $\rho_{G\setminus a}(X\cup\{c\})\leq\rho_{G}(X\cup\{a,c\},Y-\{c\})\leq\rho_{G}(X\cup\{a\},Y-\{c\})=\rho_{G\setminus c}(Y-\{c\})<3$. Therefore, $\rho_{G\setminus a}(X\cup\{c\})\leq 2=\rho_{G\setminus a}(X)$. Since $|Y-\{c\}|\geq 3$ and $G\setminus a$ is prime, we have $\rho_{G\setminus a}(X\cup\{c\})=2$. Since $Y$ is not sequential in $G\setminus a$, by Lemma~\ref{lem:basic_sequential}, $Y-\{c\}$ is not sequential in $G\setminus a$ and therefore $|Y-\{c\}|\geq 4$.
Since $T\subseteq X\cup\{a,c\}$, by~Lemma~\ref{lem:cor_s1}, $\rho_{G}(X\cup\{a,c\})=\rho_{G\setminus a}(X\cup\{c\})=2$, contradicting the assumption that $G$ is internally $3$-rank-connected. Therefore $\rho_{G\setminus c}(Y-\{c\})=3$. By symmetry, we deduce that $\rho_{G\setminus b}(X-\{b\})=3$.
\end{proof}

\begin{lemma}
\label{lem:series}
Let $G$ be an internally $3$-rank-connected graph with $|V(G)|\geq 12$ and $T=\{a,b,c\}$ be a triplet of $G$ such that $G\setminus c$ is not sequentially $3$-rank-connected.
Let $X$ be a subset of $V(G\setminus a\setminus b)$ such that $|X|\geq 3$, $|V(G\setminus a\setminus b)-X|\geq 2$, and $c\notin X$. Then $\rho_{G\setminus a\setminus b}(X)\geq 2$.
\end{lemma}
\begin{proof}
Suppose that $\rho_{G\setminus a\setminus b}(X)\leq 1$.
Let $Y=V(G\setminus a\setminus b)-X$.
Since $\{a,b,c\}$ is a triplet of~$G$, we have $\rho_{G\setminus a}(\{b,c\})=\rho_{G}(\{a,b,c\})$.
By~Lemma~\ref{lem:cor_s1}, $\rho_{G}(Y\cup\{a,b\})=\rho_{G\setminus a}(Y\cup\{b\})$.
Hence $\rho_{G}(Y\cup\{a,b\})=\rho_{G\setminus a}(Y\cup\{b\})
\leq\rho_{G\setminus a\setminus b}(Y)+1=\rho_{G\setminus a\setminus b}(X)+1\leq 2$. So $|X|\leq 3$ because $G$ is internally $3$-rank-connected and $|Y\cup\{a,b\}|\geq 4$.

Since $G\setminus c$ is not sequentially $3$-rank-connected, there exists a partition $(C_{a},C_{b})$ of $V(G\setminus c)$ such that $\rho_{G\setminus c}(C_{a})\leq 2$ and neither $C_{a}$ nor $C_{b}$ is sequential in $ G\setminus c$.

Suppose that $|X|=3$. By symmetry, we may assume that $|C_{a}\cap X|\geq 2$ by swapping $C_{a}$ and $C_{b}$ if necessary. If $|C_{a}\cap X|=2$, then let $x$~be the element in $C_{b}\cap X$. By Lemma~\ref{lem:prime}, $G\setminus c$ is prime. Since $|(Y\cup\{a,b\})-\{c\}|\geq 2$, we have $2\leq\rho_{G\setminus c}(X)\leq\rho_{G}(X)=2$. Since $|X-\{x\}|=2$ and $G\setminus c$ is prime, we also have $\rho_{G\setminus c}(X-\{x\})=2$. So by Lemma~\ref{lem:subeq},
\[
\rho_{G\setminus c}(C_{a})+\rho_{G\setminus c}(X)\geq\rho_{G\setminus c}(C_{a}\cup\{x\})+\rho_{G\setminus c}(X-\{x\}).
\]
Therefore, $\rho_{G\setminus c}(C_{a}\cup\{x\})\leq\rho_{G\setminus c}(C_{a})\leq 2$. Since $|C_{b}-\{x\}|\geq 3$ and by Lemma~\ref{lem:prime}, $G\setminus c$ is prime, we have $\rho_{G\setminus c}(C_{a}\cup\{x\})=\rho_{G\setminus c}(C_{a})=2$. So by Lemma~\ref{lem:basic_sequential}, neither $C_{a}\cup\{x\}$ nor $C_{b}-\{x\}$ is sequential in $G\setminus c$. By replacing $(C_{a},C_{b})$ with $(C_{a}\cup\{x\},C_{b}-\{x\})$, we may assume that $|C_{a}\cap X|=3$. 

By Lemma~\ref{lem:bc}, there is a unique element $t\in\{a,b\}$ of $C_{b}\cap T$.
Then $X\subseteq C_{a}$ and $C_{b}-\{t\}\subseteq Y-\{c\}\subseteq Y$.
Since $|V(G)|\geq 12$ and $G$ is internally $3$-rank-connected, we have $\rho_{G}(Y\cup\{t\})\geq 3$. Since $\rho_{G\setminus t}(Y)\leq\rho_{G\setminus a\setminus b}(Y)+1\leq 2<\rho_{G}(Y\cup\{t\})$ and $t\in C_{b}\subseteq Y\cup\{t\}$, by Lemma~\ref{lem:cor_s1}, 
$\rho_{G\setminus t}(C_{b}-\{t\})<\rho_{G}(C_{b})\leq 3$, contradicting Lemma~\ref{lem:bc}.
\end{proof}

\begin{lemma}
\label{lem:eq}
Let $G$ be an internally $3$-rank-connected graph with $|V(G)|\geq 12$ and $T=\{a,b,c\}$ be a triplet of $G$. Let $(A_{b}, A_{c})$ be a partition of $V(G\setminus a)$ such that $b\in A_{b}$, $c\in A_{c}$, $\rho_{G\setminus a}(A_{b})\leq 2$, and neither $A_{b}$ nor $A_{c}$ is sequential in $G\setminus a$ and let $(B_{a}, B_{c})$ be a partition of $V(G\setminus b)$ such that $a\in B_{a}$, $c\in B_{c}$, $\rho_{G\setminus b}(B_{a})\leq 2$, and neither $B_{a}$ nor $B_{c}$ is sequential in $G\setminus b$. If $G\setminus c$ is not sequentially $3$-rank-connected, then the following hold.
\begin{enumerate}[label=\rm(\arabic*)]
\item $\rho_{G\setminus a\setminus b}(A_{b}\cap B_{c})=\rho_{G}(A_{b}\cap B_{c})$.
\item $\rho_{G\setminus a\setminus b}(A_{c}\cap B_{a})=\rho_{G}(A_{c}\cap B_{a})$.
\item $\rho_{G\setminus a\setminus b}(A_{c}\cap B_{c})=\rho_{G}(A_{c}\cap B_{c})$.
\end{enumerate}
\end{lemma}

\begin{proof}
Since none of $A_{b}$, $A_{c}$ is sequential in $G\setminus a$ and none of $B_{a}$, $B_{c}$ is sequential in $G\setminus b$, we have $|A_{b}|,|A_{c}|,|B_{a}|,|B_{c}|\geq 4$.
By Lemma~\ref{lem:prime}, $G\setminus a$ is prime and so $\rho_{G\setminus a}(A_{c})=2$. Since $c\notin A_{b}-\{b\}$ and $|A_{b}-\{b\}|\geq 3$, by Lemma~\ref{lem:series}, we have $\rho_{G\setminus a\setminus b}(A_{c})=\rho_{G\setminus a\setminus b}(A_{b}-\{b\})\geq 2$. So by Lemma~\ref{lem:delrank}(i), we have $\rho_{G\setminus a\setminus b}(A_{c})=\rho_{G\setminus a}(A_{c})=2$. Similarly, $\rho_{G\setminus a\setminus b}(B_{c})=\rho_{G\setminus b}(B_{c})=2$.

Since $\rho_{G\setminus a\setminus b}(B_{c})=\rho_{G\setminus b}(B_{c})=2$ and $A_{b}\cap B_{c}\subseteq B_{c}$, by Lemma~\ref{lem:cor_s2}, we have 
$\rho_{G\setminus b}(A_{b}\cap B_{c})=\rho_{G\setminus a\setminus b}(A_{b}\cap B_{c})$. 

Since $\{a,b,c\}$ is a triplet of $G$, we have $\rho_{G}(\{a,b,c\})=\rho_{G\setminus b}(\{a,c\})$. 
Observe that $\rho_{G\setminus b}(A_{b}\cap B_{c})=\rho_{G\setminus b}(A_{c}\cup B_{a})$ and $\rho_{G}(A_{b}\cap B_{c})=\rho_{G}(A_{c}\cup B_{a}\cup\{b\})$. Since $\{a,b,c\}\subseteq A_{c}\cup B_{a}\cup\{b\}$, by~Lemma~\ref{lem:cor_s1},
$\rho_{G\setminus b}(A_{b}\cap B_{c})=\rho_{G\setminus b}(A_{c}\cup B_{a})=\rho_{G}(A_{c}\cup B_{a}\cup\{b\})=\rho_{G}(A_{b}\cap B_{c})$. 

Hence $\rho_{G\setminus a\setminus b}(A_{b}\cap B_{c})=\rho_{G}(A_{b}\cap B_{c})$ and (1) holds. By symmetry, (2) also holds. 

Now let us prove (3). Since $\rho_{G\setminus a\setminus b}(B_{c})=\rho_{G\setminus b}(B_{c})=2$ and $A_{c}\cap B_{c}\subseteq B_{c}$, by Lemma~\ref{lem:cor_s2}, we have $\rho_{G\setminus a\setminus b}(A_{c}\cap B_{c})=\rho_{G\setminus b}(A_{c}\cap B_{c})$.

By~\ref{item:ab1} of Lemma~\ref{lem:sub_eq_AB}, 
\begin{align*}
2+\rho_{G\setminus b}(A_{c}\cap B_{c})&=\rho_{G\setminus a}(A_{c})+\rho_{G\setminus b}(A_{c}\cap B_{c}) \\
&\geq\rho_{G\setminus a\setminus b}(A_{c})+\rho_{G}(A_{c}\cap B_{c})=2+\rho_{G}(A_{c}\cap B_{c}),
\end{align*}
which implies that $\rho_{G\setminus b}(A_{c}\cap B_{c})\geq \rho_{G}(A_{c}\cap B_{c})$. By (i) of Lemma~\ref{lem:delrank}, $\rho_{G\setminus b}(A_{c}\cap B_{c})\leq \rho_{G}(A_{c}\cap B_{c})$ and so $\rho_{G\setminus b}(A_{c}\cap B_{c})= \rho_{G}(A_{c}\cap B_{c})$.
Hence $\rho_{G}(A_{c}\cap B_{c})=\rho_{G\setminus a\setminus b}(A_{c}\cap B_{c})$.
\end{proof}

\begin{lemma}
\label{lem:lessthan2}
Let $G$ be an internally $3$-rank-connected graph with $|V(G)|\geq 12$ and $T=\{a,b,c\}$ be a triplet of $G$. Let $(A_{b}, A_{c})$ be a partition of $V(G\setminus a)$ such that $b\in A_{b}$, $c\in A_{c}$, $\rho_{G\setminus a}(A_{b})\leq 2$, and neither $A_{b}$ nor $A_{c}$ is sequential in $G\setminus a$ and let $(B_{a}, B_{c})$ be a partition of $V(G\setminus b)$ such that $a\in B_{a}$, $c\in B_{c}$, $\rho_{G\setminus b}(B_{a})\leq 2$, and neither $B_{a}$ nor $B_{c}$ is sequential in $G\setminus b$. If $G\setminus c$ is not sequentially $3$-rank-connected, then the following hold.
\begin{enumerate}[label=\rm(\roman*)]
\item $\rho_{G}(A_{c}\cap B_{a})\leq 2$ and $2\leq |A_{c}\cap B_{a}|\leq 3$.
\item $\rho_{G}(A_{b}\cap B_{c})\leq 2$ and $2\leq |A_{b}\cap B_{c}|\leq 3$.
\item $\rho_{G\setminus a\setminus b}(A_{b}\cap B_{a})\leq 2$.
\item $|A_{c}\cap B_{c}|\geq 2$.
\item If $\rho_{G\setminus a\setminus b}(A_{b}\cap B_{a})\geq 2$, then $\rho_{G}(A_{c}\cap B_{c})\leq 2$ and $|A_{c}\cap B_{c}|\leq 3$.
\end{enumerate}
\end{lemma}

\begin{proof}
Since none of $A_{b}$, $A_{c}$ is sequential in $G\setminus a$ and none of $B_{a}$, $B_{c}$ is sequential in $G\setminus b$, we have $|A_{b}|,|A_{c}|,|B_{a}|,|B_{c}|\geq 4$.
Let us prove the following, which prove the lemma.
\begin{enumerate}[label=\rm(\arabic*)]
\item If $|A_{b}\cap B_{c}|\geq 2$, then $\rho_{G}(A_{c}\cap B_{a})\leq 2$ and $|A_{c}\cap B_{a}|\leq 3$.
\item If $|A_{c}\cap B_{a}|\geq 2$, then $\rho_{G}(A_{b}\cap B_{c})\leq 2$ and $|A_{b}\cap B_{c}|\leq 3$.
\item If $|A_{c}\cap B_{c}|\geq 2$, then $\rho_{G\setminus a\setminus b}(A_{b}\cap B_{a})\leq 2$.
\item If $\rho_{G\setminus a\setminus b}(A_{b}\cap B_{a})\geq 2$, then $\rho_{G}(A_{c}\cap B_{c})\leq 2$ and $|A_{c}\cap B_{c}|\leq 3$.
\item $|A_{b}\cap B_{c}|\geq 2$.
\item $|A_{c}\cap B_{a}|\geq 2$.
\item $|A_{c}\cap B_{c}|\geq 2$.
\end{enumerate}

To prove (1), suppose that $|A_{b}\cap B_{c}|\geq 2$.
Since $G$ is prime and $|V(G)-(A_{b}\cap B_{c})|\geq |A_{c}|\geq 4$, by (1) of Lemma~\ref{lem:eq}, $\rho_{G\setminus a\setminus b}(A_{b}\cap B_{c})=\rho_{G}(A_{b}\cap B_{c})\geq 2$. Since $G\setminus b$ is prime and $|A_{b}\cap B_{c}|\geq 2$, we have $\rho_{G\setminus b}(A_{c}\cup B_{a})=\rho_{G\setminus b}(A_{b}\cap B_{c})\geq 2$. Since $\rho_{G\setminus a\setminus b}(A_{c})=2$, by~\ref{item:s1} of Lemma~\ref{lem:subtool},
\begin{align*}
2+2&=\rho_{G\setminus a\setminus b}(A_{c})+\rho_{G\setminus b}(B_{a}) \\
&\geq\rho_{G\setminus a\setminus b}(A_{c}\cap B_{a})+\rho_{G\setminus b}(A_{c}\cup B_{a})\geq \rho_{G\setminus a\setminus b}(A_{c}\cap B_{a})+2.
\end{align*}
Therefore, by (2) of Lemma~\ref{lem:eq}, $\rho_{G}(A_{c}\cap B_{a})=\rho_{G\setminus a\setminus b}(A_{c}\cap B_{a})\leq 2$. Since $G$ is internally $3$-rank-connected and $|V(G)-(A_{c}\cap B_{a})|\geq |A_{b}|\geq 4$, we deduce that $|A_{c}\cap B_{a}|\leq 3$.
So this proves (1). By symmetry between $a$ and~$b$, (2) also holds. 

Now we show (3). Suppose that $|A_{c}\cap B_{c}|\geq 2$. Since $G$ is prime and $|V(G)-(A_{b}\cup B_{a})|\geq |A_{c}|\geq 4$, we have $\rho_{G}(A_{b}\cup B_{a})\geq 2$. By~\ref{item:ab3} of Lemma~\ref{lem:sub_eq_AB}, 
\[
4\geq \rho_{G\setminus a}(A_{b})+\rho_{G\setminus b}(B_{a})\geq \rho_{G}(A_{b}\cup B_{a})+\rho_{G\setminus a\setminus b}(A_{b}\cap B_{a})
\]
and therefore $\rho_{G\setminus a\setminus b}(A_{b}\cap B_{a})\leq 2$.

Now let us prove (4). Suppose that $\rho_{G\setminus a\setminus b}(A_{b}\cap B_{a})\geq 2$. By~\ref{item:ab3} of Lemma~\ref{lem:sub_eq_AB},
\[
4\geq \rho_{G\setminus a}(A_{b})+\rho_{G\setminus b}(B_{a})\geq\rho_{G}(A_{b}\cup B_{a})+\rho_{G\setminus a\setminus b}(A_{b}\cap B_{a})\geq\rho_{G}(A_{b}\cup B_{a})+2.
\]
Hence $\rho_{G}(A_{b}\cup B_{a})=\rho_{G}(A_{c}\cap B_{c})\leq 2$.
Since $G$ is internally $3$-rank-connected and $|V(G)-(A_{c}\cap B_{c})|\geq 4$, we conclude that $|A_{c}\cap B_{c}|\leq 3$.

To prove (5), suppose that $|A_{b}\cap B_{c}|\leq 1$. Then $4\leq|A_{b}|=|\{b\}|+|A_{b}\cap B_{c}|+|A_{b}\cap B_{a}|\leq 2+|A_{b}\cap B_{a}|$ and so $|A_{b}\cap B_{a}|\geq 2$. 

If $|A_{b}\cap B_{a}|\geq 3$, then, since $c\in A_{c}\cap B_{c}$, by Lemma~\ref{lem:series}, $\rho_{G\setminus a\setminus b}(A_{b}\cap B_{a})\geq 2$.
If $|A_{b}\cap B_{a}|=2$, then $|A_{b}|=4$ and by Lemma~\ref{lem:quad}, $A_{b}$ is a quad of $G\setminus a$. Then by (ii) of Lemma~\ref{lem:delrank}, $\rho_{G\setminus a\setminus b}(A_{b}\cap B_{a})\geq\rho_{G\setminus a}((A_{b}\cap B_{a})\cup\{b\})-1=2$. So, in both cases, we deduce that $\rho_{G\setminus a\setminus b}(A_{b}\cap B_{a})\geq 2$.

Hence, by (4), $\rho_{G}(A_{c}\cap B_{c})\leq 2$ and $|A_{c}\cap B_{c}|\leq 3$. 
Since $4\leq |B_{c}|=|A_{b}\cap B_{c}|+|A_{c}\cap B_{c}|\leq 1+|A_{c}\cap B_{c}|\leq 4$, we have $|A_{c}\cap B_{c}|=3$ and $|B_{c}|=4$. By (i) of Lemma~\ref{lem:delrank}, $\rho_{G\setminus b}(A_{c}\cap B_{c})\leq \rho_{G}(A_{c}\cap B_{c})\leq 2$. So $B_{c}$ is sequential in $G\setminus b$, contradicting our assumption.
So this proves that $|A_{b}\cap B_{c}|\geq 2$ and by symmetry between $a$ and $b$, $|A_{c}\cap B_{a}|\geq 2$ and (6) holds. 

Now let us prove (7). Suppose that $|A_{c}\cap B_{c}|\leq 1$. Then $4\leq|A_{c}|=1+|A_{c}\cap B_{c}|+|A_{c}\cap B_{a}|\leq 2+|A_{c}\cap B_{a}|$ and so $2\leq|A_{c}\cap B_{a}|$. Then by (2), we have $\rho_{G}(A_{b}\cap B_{c})\leq 2$ and $|A_{b}\cap B_{c}|\leq 3$. 
Since $4\leq |B_{c}|=|A_{b}\cap B_{c}|+|A_{c}\cap B_{c}|\leq |A_{b}\cap B_{c}|+1\leq 4$, we have $|A_{b}\cap B_{c}|=3$ and $|B_{c}|=4$. By (i) of Lemma~\ref{lem:delrank}, $\rho_{G\setminus b}(A_{b}\cap B_{c})\leq \rho_{G}(A_{b}\cap B_{c})\leq 2$. So $B_{c}$ is sequential in $G\setminus b$, contradicting our assumption.
\end{proof}

\begin{lemma}
\label{lem:acbc}
Let $G$ be an internally $3$-rank-connected graph with $|V(G)|\geq 12$ and $T=\{a,b,c\}$ be a triplet of $G$. Let $(A_{b}, A_{c})$ be a partition of $V(G\setminus a)$ such that $b\in A_{b}$, $c\in A_{c}$, $\rho_{G\setminus a}(A_{b})\leq 2$, and neither $A_{b}$ nor $A_{c}$ is sequential in $G\setminus a$, let $(B_{a}, B_{c})$ be a partition of $V(G\setminus b)$ such that $a\in B_{a}$, $c\in B_{c}$, $\rho_{G\setminus b}(B_{a})\leq 2$, and neither $B_{a}$ nor $B_{c}$ is sequential in $G\setminus b$, and let $(C_{a}, C_{b})$ be a partition of $V(G\setminus c)$ such that $a\in C_{a}$, $b\in C_{b}$, $\rho_{G\setminus c}(C_{a})\leq 2$, and neither $C_{a}$ nor $C_{b}$ is sequential in $G\setminus c$. Then the following hold.
\begin{enumerate}[label=\rm(\arabic*)]
\item If $|A_{c}\cap B_{c}|\geq 3$ and $\rho_{G\setminus a\setminus b}(A_{b}\cap B_{a})>1$, then $|A_{c}\cap B_{c}|=~3$, $\rho_{G}(A_{c}\cap B_{c})=2$, and $|A_{c}\cap B_{c}\cap C_{a}|=|A_{c}\cap B_{c}\cap C_{b}|=1$.
\item If $|A_{c}\cap B_{c}|\geq 3$ and $\rho_{G\setminus a\setminus b}(A_{b}\cap B_{a})\leq 1$, then either
\begin{itemize}
\item $A_{b}\cap B_{a}=\emptyset$, or 
\item $1\leq |A_{b}\cap B_{a}|\leq 2$ and $\rho_{G\setminus c}((A_{b}\cap B_{a})\cup\{a,b\})=~3$.
\end{itemize}
\end{enumerate}
\end{lemma}

\begin{proof}
(1)
Since $\rho_{G\setminus a\setminus b}(A_{b}\cap B_{a})>1$, we have $|A_{b}\cap B_{a}|\geq 2$ and by (v) of Lemma~\ref{lem:lessthan2}, $\rho_{G}(A_{c}\cap B_{c})\leq 2$ and $|A_{c}\cap B_{c}|\leq~3$. Hence $|A_{c}\cap B_{c}|=3$. Since $G$ is prime and $|V(G)-(A_c\cap B_c)|\geq 3$, we have $\rho_{G}(A_{c}\cap B_{c})=2$.
Now we prove that $|A_{c}\cap B_{c}\cap C_{a}|=|A_{c}\cap B_{c}\cap C_{b}|=1$.
Suppose not. Then, by symmetry, we may assume that $|A_{c}\cap B_{c}\cap C_{a}|=2$ and $|A_{c}\cap B_{c}\cap C_{b}|=0$. Since $|(A_{c}\cap B_{c})-\{c\}|=2$ and $G\setminus c$ is prime, $\rho_{G\setminus c}((A_{c}\cap B_{c})-\{c\})\geq 2$. By (ii) of Lemma~\ref{lem:delrank}, $\rho_{G\setminus c}((A_{c}\cap B_{c})-\{c\})=\rho_{G}(A_{c}\cap B_{c})=2$.
Since $A_{c}\cap B_{c}\subseteq C_{a}\cup\{c\}$, by~Lemma~\ref{lem:cor_s1}, 
$\rho_{G}(C_{a}\cup\{c\})=\rho_{G\setminus c}(C_{a})\leq 2$. Since $G$ is internally $3$-rank-connected and $|C_{a}\cup\{c\}|\geq 5$, we have $|C_{b}|\leq 3$, contradicting our assumption.

\medskip 
\noindent
(2) By Lemma~\ref{lem:series}, $|A_{b}\cap B_{a}|\leq 2$. Suppose that $|A_{b}\cap B_{a}|\geq 1$. We can observe that $\rho_{G}((A_{b}\cap B_{a})\cup\{a,b,c\})\geq 3$ because $G$ is internally $3$-rank-connected, $|(A_{b}\cap B_{a})\cup\{a,b,c\}|\geq 4$, and $|V(G)-((A_{b}\cap B_{a})\cup\{a,b,c\})|\geq 12-5=7$. Since $\{a,b,c\}$ is a triplet of $G$, we have $\rho_{G}(\{a,b,c\})=\rho_{G\setminus c}(\{a,b\})=2$. Since $\{a,b,c\}\subseteq (A_{b}\cap B_{a})\cup\{a,b,c\}$, by~Lemma~\ref{lem:cor_s1}, $\rho_{G\setminus c}((A_{b}\cap B_{a})\cup\{a,b\})=\rho_{G}((A_{b}\cap B_{a})\cup\{a,b,c\})\geq 3$. By (i) and (ii) of Lemma~\ref{lem:delrank}, $\rho_{G\setminus c}((A_{b}\cap B_{a})\cup\{a,b\})\leq\rho_{G}((A_{b}\cap B_{a})\cup\{a,b\})\leq 2+\rho_{G\setminus a\setminus b}(A_{b}\cap B_{a})\leq 3$ and we conclude that $\rho_{G\setminus c}((A_{b}\cap B_{a})\cup\{a,b\})=3$.
\end{proof}

\begin{proposition}
\label{prop:internal}
Let $T$ be a triplet of an internally $3$-rank-connected graph $G$. If $|V(G)|\geq 12$, then there exists $t\in T$ such that $G\setminus t$ is sequentially $3$-rank-connected.
\end{proposition}

\begin{proof}
Let $T=\{a,b,c\}$.
Suppose that none of $G\setminus a$, $G\setminus b$, and $G\setminus c$ is sequentially $3$-rank-connected. Then there exist partitions $(A_{b},A_{c})$ of $V(G)-\{a\}$, $(B_{a},B_{c})$ of $V(G)-\{b\}$, and $(C_{a},C_{b})$ of $V(G)-\{c\}$ such that $\rho_{G\setminus a}(A_{b})\leq 2$, $\rho_{G\setminus b}(B_{a})\leq 2$, $\rho_{G\setminus c}(C_{a})\leq 2$, neither $A_{b}$ nor $A_{c}$ is sequential in $G\setminus a$, neither $B_{a}$ nor $B_{c}$ is sequential in $G\setminus b$, and neither $C_{a}$ nor $C_{b}$ is sequential in $G\setminus c$. Then 
$|A_{b}|,|A_{c}|\geq 4$, $|B_{a}|,|B_{c}|\geq 4$, and $|C_{a}|,|C_{b}|\geq 4$.

By Lemma~\ref{lem:bc}, we may assume that 
$b\in A_{b}$, $c\in A_{c}$,
$a\in B_{a}$, $c\in B_{c}$,
$a\in C_{a}$, and $b\in C_{b}$.

By Lemma~\ref{lem:lessthan2}, we have $|A_{b}\cap B_{c}|\leq 3,|A_{c}\cap B_{a}|\leq 3$, and $\rho_{G}(A_{b}\cap B_{c})\leq 2$. By symmetry between $b$ and $c$, we have that $|A_{c}\cap C_{b}|\leq 3$ and $|A_{b}\cap C_{a}|\leq 3$. By symmetry between $a$ and $c$, we have that $|B_{c}\cap C_{a}|\leq 3$ and $|B_{a}\cap C_{b}|\leq 3$.
Now we show that we can assume the following.

\begin{enumerate}[label=\rm(B\arabic*)]
\item\label{item:b1} If $|A_{b}\cap B_{c}|=3$, then $A_{b}\cap B_{c}\subseteq C_{a}$ or $A_{b}\cap B_{c}\subseteq C_{b}$.
\item\label{item:b2} If $|A_{c}\cap B_{a}|=3$, then $A_{c}\cap B_{a}\subseteq C_{a}$ or $A_{c}\cap B_{a}\subseteq C_{b}$.
\item\label{item:b3} If $|A_{c}\cap C_{b}|=3$, then $A_{c}\cap C_{b}\subseteq B_{a}$ or $A_{c}\cap C_{b}\subseteq B_{c}$.
\item\label{item:b4} If $|A_{b}\cap C_{a}|=3$, then $A_{b}\cap C_{a}\subseteq B_{a}$ or $A_{b}\cap C_{a}\subseteq B_{c}$.
\item\label{item:b5} If $|B_{c}\cap C_{a}|=3$, then $B_{c}\cap C_{a}\subseteq A_{b}$ or $B_{c}\cap C_{a}\subseteq A_{c}$.
\item\label{item:b6} If $|B_{a}\cap C_{b}|=3$, then $B_{a}\cap C_{b}\subseteq A_{b}$ or $B_{a}\cap C_{b}\subseteq A_{c}$.
\end{enumerate}

We choose $(A_{b},A_{c},B_{a},B_{c},C_{a},C_{b})$ such that $b\in A_{b}$, $c\in A_{c}$,
$a\in B_{a}$, $c\in B_{c}$,
$a\in C_{a}$, $b\in C_{b}$, and it satisfies the maximum number of ~\ref{item:b1}--\ref{item:b6}. Then we claim that all of \ref{item:b1}--\ref{item:b6} hold. Suppose not. Then by symmetry, we can assume that \ref{item:b1} does not hold. Then $|A_{b}\cap B_{c}|=3$, $A_{b}\cap B_{c}\nsubseteq C_{a}$, and $A_{b}\cap B_{c}\nsubseteq C_{b}$. 
Then either $|A_{b}\cap B_{c}\cap C_{a}|=2$ and $|A_{b}\cap B_{c}\cap C_{b}|=1$
or $|A_{b}\cap B_{c}\cap C_{a}|=1$ and $|A_{b}\cap B_{c}\cap C_{b}|=2$. 

\medskip\noindent
\textit{(i)} Suppose that $|A_{b}\cap B_{c}\cap C_{a}|=2$ and $|A_{b}\cap B_{c}\cap C_{b}|=1$.
Let $x$ be the element of $A_{b}\cap B_{c}\cap C_{b}$. We have $\rho_{G\setminus c}(A_{b}\cap B_{c})\leq \rho_{G}(A_{b}\cap B_{c})\leq 2$. Since $|(A_{b}\cap B_{c})-\{x\}|=2$ and $G\setminus c$ is prime, $\rho_{G\setminus c}((A_{b}\cap B_{c})-\{x\})\geq 2$. So by Lemma~\ref{lem:subeq}, 
\begin{align*}
2+2&\geq\rho_{G\setminus c}(C_{a})+\rho_{G\setminus c}(A_{b}\cap B_{c}) \\
&\geq\rho_{G\setminus c}((A_{b}\cap B_{c})-\{x\})+\rho_{G\setminus c}(C_{a}\cup\{x\})\geq 2+\rho_{G\setminus c}(C_{a}\cup\{x\}).
\end{align*}
Therefore, $\rho_{G\setminus c}(C_{a}\cup\{x\})\leq \rho_{G\setminus c}(C_{a})\leq 2$. Since $G\setminus c$ is prime and $|V(G\setminus c)-(C_{a}\cup\{x\})|=|C_{b}|-1\geq 3$, we have $\rho_{G\setminus c}(C_{a}\cup\{x\})=\rho_{G\setminus c}(C_{a})=2$. Hence by Lemma~\ref{lem:basic_sequential}, neither $C_{a}\cup\{x\}$ nor $C_{b}-\{x\}$ is sequential in $G\setminus c$. We deduce that $(A_{b},A_{c},B_{a},B_{c},C_{a}\cup\{x\}, C_{b}-\{x\})$ satisfies \ref{item:b1}.  
Since $x\notin A_{c}\cap B_{a}$, if $(A_{b},A_{c},B_{a},B_{c},C_{a},C_{b})$ satisfies \ref{item:b2}, then $(A_{b},A_{c},B_{a},B_{c},C_{a}\cup\{x\}, C_{b}-\{x\})$ satisfies~\ref{item:b2}.
Since $A_{c}\cap C_{b}=A_{c}\cap (C_{b}-\{x\})$, if
$(A_{b},A_{c},B_{a},B_{c},C_{a},C_{b})$ satisfies \ref{item:b3}, then $(A_{b},A_{c},B_{a},B_{c},C_{a}\cup\{x\}, C_{b}-\{x\})$ satisfies~\ref{item:b3}. Since $B_{a}\cap(C_{b}-\{x\})=B_{a}\cap C_{b}$, if $(A_{b},A_{c},B_{a},B_{c},C_{a},C_{b})$ satisfies \ref{item:b6}, then $(A_{b},A_{c},B_{a},B_{c},C_{a}\cup\{x\}, C_{b}-\{x\})$ satisfies~\ref{item:b6}. Since $x\in A_{b}$, we have $|A_{b}\cap C_{a}|+1=|A_{b}\cap (C_{a}\cup\{x\})|\leq 3$ by applying Lemma~\ref{lem:lessthan2}(i) with $(A_{c},A_{b})$ and $(C_{a}\cup\{x\}, C_{b}-\{x\})$. So $|A_{b}\cap C_{a}|\leq 2$. Since $|A_{b}\cap B_{c}\cap C_{a}|=2$ we have $A_{b}\cap C_{a}\subseteq B_{c}$.
So $A_{b}\cap (C_{a}\cup\{x\})\subseteq B_{c}$ because $x\in B_{c}$. Hence $(A_{b},A_{c},B_{a},B_{c},C_{a}\cup\{x\}, C_{b}-\{x\})$ satisfies \ref{item:b4}.

Since $x\in B_{c}$, we have $|B_{c}\cap C_{a}|+1=|B_{c}\cap (C_{a}\cup\{x\})|\leq 3$ by applying Lemma~\ref{lem:lessthan2}(ii) with $(B_{c},B_{a})$ and $(C_{b}-\{x\}, C_{a}\cup\{x\})$. So $|B_{c}\cap C_{a}|\leq 2$. Since $|A_{b}\cap B_{c}\cap C_{a}|=2$ we have $B_{c}\cap C_{a}\subseteq A_{b}$.
So $B_{c}\cap (C_{a}\cup\{x\})\subseteq A_{b}$ because $x\in A_{b}$. Hence $(A_{b},A_{c},B_{a},B_{c},C_{a}\cup\{x\}, C_{b}-\{x\})$ satisfies \ref{item:b5}.
Therefore, the number of \ref{item:b1}--\ref{item:b6} which $(A_{b},A_{c},B_{a},B_{c},C_{a}\cup\{x\}, C_{b}-\{x\})$ satisfies is larger than the number of \ref{item:b1}--\ref{item:b6} which $(A_{b},A_{c},B_{a},B_{c},C_{a}, C_{b})$ satisfies, contradicting our assumption.

\medskip\noindent
\textit{(ii)} Suppose that $|A_{b}\cap B_{c}\cap C_{a}|=1$ and $|A_{b}\cap B_{c}\cap C_{b}|=2$. Let $y$ be the element of $A_{b}\cap B_{c}\cap C_{a}$. Since $|(A_{b}\cap B_{c})-\{y\}|=2$ and $G\setminus c$ is prime, $\rho_{G\setminus c}((A_{b}\cap B_{c})-\{y\})\geq 2$. So by Lemma~\ref{lem:subeq}, 
\begin{align*}
2+2&\geq\rho_{G\setminus c}(C_{b})+\rho_{G\setminus c}(A_{b}\cap B_{c}) \\
&\geq\rho_{G\setminus c}((A_{b}\cap B_{c})-\{y\})+\rho_{G\setminus c}(C_{b}\cup\{y\})\geq 2+\rho_{G\setminus c}(C_{b}\cup\{y\}). 
\end{align*}
Therefore, $\rho_{G\setminus c}(C_{b}\cup\{y\})\leq \rho_{G\setminus c}(C_{b})\leq 2$. Since $G\setminus c$ is prime and $|V(G\setminus c)-(C_{b}\cup\{y\})|=|C_{a}|-1\geq 3$, we have $\rho_{G\setminus c}(C_{b}\cup\{y\})=\rho_{G\setminus c}(C_{b})=2$. Hence by Lemma~\ref{lem:basic_sequential}, neither $C_{a}-\{y\}$ nor $C_{b}\cup\{y\}$ is sequential in $G\setminus c$. We deduce that $(A_{b},A_{c},B_{a},B_{c},C_{a}-\{y\}, C_{b}\cup\{y\})$ satisfies \ref{item:b1}. Since $y\notin A_{c}\cap B_{a}$, if $(A_{b},A_{c},B_{a},B_{c},C_{a},C_{b})$ satisfies \ref{item:b2}, then $(A_{b},A_{c},B_{a},B_{c},C_{a}-\{y\}, C_{b}\cup\{y\})$ satisfies~\ref{item:b2}. Since $y\in A_{b}\cap C_{a}$, by applying Lemma~\ref{lem:lessthan2}(i) with $(A_{c},A_{b})$ and $(C_{a},C_{b})$, we have $|A_{b}\cap (C_{a}-\{y\})|=|A_{b}\cap C_{a}|-1\leq 3-1=2$. Hence $(A_{b},A_{c},B_{a},B_{c},C_{a}-\{y\}, C_{b}\cup\{y\})$ satisfies \ref{item:b4}. Since $y\in B_{c}\cap C_{a}$, by applying Lemma~\ref{lem:lessthan2}(ii) with $(B_{c},B_{a})$ and $(C_{b},C_{a})$, we have $|B_{c}\cap (C_{a}-\{y\})|=|B_{c}\cap C_{a}|-1\leq 3-1=2$. 
Hence $(A_{b},A_{c},B_{a},B_{c},C_{a}-\{y\}, C_{b}\cup\{y\})$ satisfies \ref{item:b5}.

Since $A_{c}\cap (C_{b}\cup\{y\})= A_{c}\cap C_{b}$, if $(A_{b},A_{c},B_{a},B_{c},C_{a},C_{b})$ satisfies \ref{item:b3}, then $(A_{b},A_{c},B_{a},B_{c},C_{a}-\{y\},C_{b}\cup\{y\})$ satisfies \ref{item:b3}.
Since $B_{a}\cap (C_{b}\cup\{y\})=B_{a}\cap C_{b}$, if $(A_{b},A_{c},B_{a},B_{c},C_{a},C_{b})$ satisfies \ref{item:b6}, then $(A_{b},A_{c},B_{a},B_{c},C_{a}-\{y\},C_{b}\cup\{y\})$ satisfies \ref{item:b6}. 
Therefore, the number of \ref{item:b1}--\ref{item:b6} which $(A_{b},A_{c},B_{a},B_{c},C_{a}-\{y\}, C_{b}\cup\{y\})$ satisfies is larger than the number of \ref{item:b1}--\ref{item:b6} which $(A_{b},A_{c},B_{a},B_{c},C_{a}, C_{b})$ satisfies, contradicting our assumption.

Therefore, the claim is proved and  $(A_{b},A_{c},B_{a},B_{c},C_{a},C_{b})$ satisfies \ref{item:b1}--\ref{item:b6}.
\begin{claim}
\label{claim:baa}
$|A_{b}\cap B_{a}\cap C_{a}|\leq 1$.
\end{claim}
\begin{subproof}
Suppose that $|A_{b}\cap B_{a}\cap C_{a}|\geq 2$. If $|A_{b}\cap C_{a}|=2$, then $A_{b}\cap C_{a}\subseteq B_{a}$ and so $A_{b}\cap B_{c}\cap C_{a}=\emptyset$. If $|A_{b}\cap C_{a}|=3$, then by~\ref{item:b4}, $A_{b}\cap B_{c}\cap C_{a}=\emptyset$. Since $2\leq |A_{b}\cap C_{a}|\leq 3$, we deduce that $A_{b}\cap B_{c}\cap C_{a}=\emptyset$.

By applying Lemma~\ref{lem:lessthan2}(ii) with $(B_{c},B_{a})$ and $(C_{b},C_{a})$, we have that $|B_{c}\cap C_{a}|\geq 2$. 
Since $A_{b}\cap B_{c}\cap C_{a}=\emptyset$, we have $|A_{c}\cap B_{c}\cap C_{a}|=|B_{c}\cap C_{a}|\geq 2$ and so $|A_{c}\cap B_{c}|\geq|\{c\}|+ |A_{c}\cap B_{c}\cap C_{a}|\geq 3$. 
Since $|A_{c}\cap B_{c}\cap C_{a}|\geq 2$, by Lemma~\ref{lem:acbc}(1), $\rho_{G\setminus a\setminus b}(A_{b}\cap B_{a})\leq 1$. So by Lemma~\ref{lem:acbc}(2), 
\[
\text{$|A_{b}\cap B_{a}|=2$ and $\rho_{G\setminus c}((A_{b}\cap B_{a})\cup\{a,b\})=3$},
\]
because $|A_{b}\cap B_{a}|\geq |A_{b}\cap B_{a}\cap C_{a}|\geq 2$. Hence $A_{b}\cap B_{a}\subseteq C_{a}$. 

By Lemma~\ref{lem:prime}, $G\setminus a$ is prime and so $\rho_{G\setminus a}(A_{b}\cap B_{a})=2$.
By (ii) of Lemma~\ref{lem:delrank}, we have $\rho_{G\setminus a}((A_{b}\cap B_{a})\cup\{b\})\leq \rho_{G\setminus a\setminus b}(A_{b}\cap B_{a})+1\leq 2$. So by~\ref{item:ab2} of Lemma~\ref{lem:sub_eq_AB}, 
\begin{align*}
\rho_{G\setminus c}((A_{b}\cap B_{a})\cup\{a\})+2&\geq\rho_{G\setminus c}((A_{b}\cap B_{a})\cup\{a\})+\rho_{G\setminus a}((A_{b}\cap B_{a})\cup\{b\}) \\
&\geq\rho_{G\setminus c}((A_{b}\cap B_{a})\cup\{a,b\})+\rho_{G\setminus a}(A_{b}\cap B_{a})=3+2,
\end{align*}
which implies that $\rho_{G\setminus c}((A_{b}\cap B_{a})\cup\{a\})\geq 3$.
Therefore, by Lemma~\ref{lem:subeq_minus}, 
\begin{align*}
3+2&\geq\rho_{G\setminus c}((A_{b}\cap B_{a})\cup\{a,b\})+\rho_{G\setminus c}(C_{b}) \\
&\geq \rho_{G\setminus c}((A_{b}\cap B_{a})\cup\{a\})+\rho_{G\setminus c}(C_{b}-\{b\})\geq 3+\rho_{G\setminus c}(C_{b}-\{b\}).
\end{align*}
Therefore, $\rho_{G\setminus c}(C_{b}-\{b\})\leq 2$. By Lemma~\ref{lem:prime}, $G\setminus c$ is prime. Since $|C_{b}-\{b\}|\geq 3$, we have $\rho_{G\setminus c}(C_{b}-\{b\})=2$. So by Lemma~\ref{lem:basic_sequential}, neither $C_{a}\cup\{b\}$ nor $C_{b}-\{b\}$ is sequential in $G\setminus c$, contradicting Lemma~\ref{lem:bc} because $\{a,b\}\subseteq C_{a}\cup\{b\}$.
\end{subproof}

Hence, by symmetry, we have $|A_{b}\cap B_{a}\cap C_{a}|\leq 1$, $|A_{c}\cap B_{a}\cap C_{a}|\leq 1$, $|A_{b}\cap B_{a}\cap C_{b}|\leq 1$, $|A_{b}\cap B_{c}\cap C_{b}|\leq 1$, $|A_{c}\cap B_{c}\cap C_{a}|\leq 1$, and $|A_{c}\cap B_{c}\cap C_{b}|\leq 1$.

\begin{claim}
\label{claim:bca}
$|A_{b}\cap B_{c}\cap C_{a}|\leq 1$.
\end{claim}
\begin{subproof}
Suppose that $|A_{b}\cap B_{c}\cap C_{a}|\geq 2$. If $|A_{b}\cap B_{c}|=2$, then $A_{b}\cap B_{c}\subseteq C_{a}$ and $A_{b}\cap B_{c}\cap C_{b}=\emptyset$. If $|A_{b}\cap B_{c}|=3$, then by~\ref{item:b1}, $A_{b}\cap B_{c}\cap C_{b}=\emptyset$. By Lemma~\ref{lem:lessthan2}(i), we have $2\leq|A_{b}\cap B_{c}|\leq 3$. So we deduce that $A_{b}\cap B_{c}\cap C_{b}=\emptyset$.

By symmetry between $(a,b,c)$ and $(c,a,b)$, we deduce that $C_{a}\cap A_{b}\cap B_{a}=~\emptyset$.
By symmetry between $(a,b,c)$ and $(b,c,a)$, we deduce that $B_{c}\cap C_{a}\cap A_{c}=\emptyset$. By Lemma~\ref{lem:lessthan2}(iv), $|A_{c}\cap B_{c}|\geq 2$. So we deduce that
\[
1 \leq |A_{c}\cap B_{c}|-|\{c\}|-|A_{c}\cap B_{c}\cap C_{a}|=|A_{c}\cap B_{c}\cap C_{b}|\leq 1,
\] 
and therefore $|A_{c}\cap B_{c}\cap C_{b}|=1$.

If $|A_{c}\cap C_{b}|=3$, then by~\ref{item:b3}, $|A_{c}\cap B_{a}\cap C_{b}|=0$. If $|A_{c}\cap C_{b}|\leq 2$, then $|A_{c}\cap B_{a}\cap C_{b}|=|A_{c}\cap C_{b}|-|A_{c}\cap B_{c}\cap C_{b}|\leq 2-1=1$. Since $|A_{c}\cap C_{b}|\leq 3$, in both cases, we deduce that $|A_{c}\cap B_{a}\cap C_{b}|\leq 1$.
Then we have
\begin{align*}
|V(G)|&=|A_{b}\cap B_{a}\cap C_{a}|+|A_{b}\cap B_{a}\cap C_{b}|+|A_{b}\cap B_{c}\cap C_{a}|+|A_{b}\cap B_{c}\cap C_{b}| \\
&\quad+|A_{c}\cap B_{a}\cap C_{a}|+|A_{c}\cap B_{a}\cap C_{b}|+|A_{c}\cap B_{c}\cap C_{a}|+|A_{c}\cap B_{c}\cap C_{b}|+|\{a,b,c\}| \\
&=0+|A_{b}\cap B_{a}\cap C_{b}|+|A_{b}\cap B_{c}\cap C_{a}|+0 \\
&\quad+|A_{c}\cap B_{a}\cap C_{a}|+|A_{c}\cap B_{a}\cap C_{b}|+0+|A_{c}\cap B_{c}\cap C_{b}|+|\{a,b,c\}| \\
&\leq 0+1+|A_{b}\cap B_{c}|+0+1+1+0+1+3\leq 10,
\end{align*}
contradicting our assumption.
\end{subproof}
By symmetry, we have $|A_{c}\cap B_{a}\cap C_{b}|\leq 1$.
Therefore, we have
\begin{align*}
|V(G)|&=|A_{b}\cap B_{a}\cap C_{a}|+|A_{b}\cap B_{a}\cap C_{b}|+|A_{b}\cap B_{c}\cap C_{a}|+|A_{b}\cap B_{c}\cap C_{b}| \\
&\quad+|A_{c}\cap B_{a}\cap C_{a}|+|A_{c}\cap B_{a}\cap C_{b}|+|A_{c}\cap B_{c}\cap C_{a}|+|A_{c}\cap B_{c}\cap C_{b}|+|\{a,b,c\}|\leq 11,
\end{align*} 
contradicting our assumption.
\end{proof}

\section{Completing the proof}
\label{sec:sequential}
A set $X$ of vertices of a graph $G$ is \emph{fully closed} if $\rho_{G}(X\cup\{v\})>\rho_{G}(X)$ for all $v\in V(G)-X$.
\begin{lemma}[Oum~{\cite[Proposition 3.1]{Oum2020}}]
\label{lem:getprime}
Let $G$ be a prime graph with $|V(G)|\geq 8$. Suppose that $G$ has a fully closed set $A$ such that $\rho_{G}(A)\geq 2$. Then there is a vertex $v$ of $A$ such that $G\setminus v$ or $G/v$ is prime.
\end{lemma}
\begin{lemma}
\label{lem:gut}
Let $G$ be a sequentially $3$-rank-connected graph and $a_{1},a_{2},\ldots, a_{k}$ be distinct vertices of $G$ such that $k\geq 4$ and $\rho_{G}(\{a_{1},\ldots,a_{i}\})\leq 2$ for each $i\leq k$. For each $1\leq i\leq k$, if $G\setminus a_{i}$ is prime, then $G\setminus a_{i}$ is sequentially $3$-rank-connected.
\end{lemma}
\begin{proof}
Since $G$ is prime, we know that $\rho_{G}(\{a_{1},\ldots,a_{j}\})=\min\{2,|V(G)|-j\}$ for each $2\leq j\leq k$. So $\rho_{G}(\{a_{1},\ldots,a_{j-1}\})\geq\rho_{G}(\{a_{1},\ldots,a_{j}\})$ for each $2\leq j\leq k$. For each $3\leq j\leq i-1$, by~\ref{item:s2} of Lemma~\ref{lem:subtool}, we have
\[
\rho_{G}(\{a_{1},\ldots,a_{j}\})+\rho_{G\setminus a_{i}}(\{a_{1},\ldots,a_{j-1}\})\geq\rho_{G}(\{a_{1},\ldots,a_{j-1}\})+\rho_{G\setminus a_{i}}(\{a_{1},\ldots,a_{j}\})
\]
and therefore $\rho_{G\setminus a_{i}}(\{a_{1},\ldots,a_{j-1}\})\geq\rho_{G\setminus a_{i}}(\{a_{1},\ldots,a_{j}\})$.

Suppose that $G\setminus a_{i}$ is prime and not sequentially $3$-rank-connected. 

Let us first consider the case when $i>3$. By Lemma~\ref{lem:contain_triple}, there is a subset $X$ of $V(G\setminus a_{i})$ such that $\rho_{G\setminus a_{i}}(X)\leq 2$, neither $X$ nor $V(G\setminus a_{i})-X$ is sequential in $G\setminus a_{i}$, and $\{a_{1},a_{2},a_{3}\}\subseteq X$. 
We may assume that $X$ is maximal among all such sets. 

We claim that $\{a_{1},\ldots,a_{i-1}\}\subseteq X$. Suppose not. Let $j\leq i-1$ be the minimum index such that $a_{j}\notin X$. Then $\{a_{1},\ldots,a_{j-1}\}\subseteq X$. Note that $j\geq 4$.
Let $Y=V(G\setminus a_{i})-X$. Since neither $X$ nor $Y$ is sequential in $G\setminus a_{i}$, we have $|X|,|Y|\geq 4$.
Since $\rho_{G\setminus a_{i}}(\{a_{1},\ldots,a_{j-1}\})\geq\rho_{G\setminus a_{i}}(\{a_{1},\ldots,a_{j}\})$, by Lemma~\ref{lem:subeq},
\[
\rho_{G\setminus a_{i}}(X)+\rho_{G\setminus a_{i}}(\{a_{1},\ldots,a_{j}\})\geq
\rho_{G\setminus a_{i}}(X\cup\{a_{j}\})+\rho_{G\setminus a_{i}}(\{a_{1},\ldots,a_{j-1}\}),
\]
and therefore $\rho_{G\setminus a_{i}}(X\cup\{a_{j}\})\leq\rho_{G\setminus a_{i}}(X)\leq 2$. Since $G\setminus a_{i}$ is prime and $|Y-\{a_{i}\}|\geq 3$, we have $\rho_{G\setminus a_{i}}(X\cup\{a_{j}\})=\rho_{G\setminus a_{i}}(X)=2$. Hence by Lemma~\ref{lem:basic_sequential}, neither $X\cup\{a_{j}\}$ nor $Y-\{a_{j}\}$ is sequential in $G\setminus a_{i}$, contradicting the maximality of $X$.
Hence $\{a_{1},\ldots,a_{i-1}\}\subseteq X$. 

Then by~\ref{item:s1} of Lemma~\ref{lem:subtool}, 
\[
\rho_{G\setminus a_{i}}(X)+\rho_{G}(\{a_{1},\ldots,a_{i}\})\geq\rho_{G}(X\cup\{a_{i}\})+\rho_{G\setminus a_{i}}(\{a_{1},\ldots,a_{i-1}\}).
\]
Since $G\setminus a_{i}$ is prime and $i>3$, we have $\rho_{G\setminus a_{i}}(\{a_{1},\ldots,a_{i-1}\})\geq\min\{2,|V(G)|-i\}=\rho_{G}(\{a_{1},\ldots,a_{i}\})$. So $\rho_{G}(X\cup\{a_{i}\})\leq\rho_{G\setminus a_{i}}(X)\leq 2$. Since $G$ is sequentially $3$-rank-connected, $X\cup\{a_{i}\}$ or $Y$ is sequential in $G$. Then by (i), (ii) of Lemma~\ref{lem:delrank}, $X$ or $Y$ is sequential in $G\setminus a_{i}$, contradicting our assumption.

Now we consider the case when $i\leq 3$. By permuting $a_{1}$, $a_{2}$, $a_{3}$, we can assume that $i=3$. Suppose that $G\setminus a_{3}$ is prime. By Lemma~\ref{lem:delrank}(ii), we have $\rho_{G\setminus a_{3}}(\{a_{1},a_{2},a_{4}\})\leq\rho_{G}(\{a_{1},a_{2},a_{3},a_{4}\})\leq 2$. Since $a_{1},a_{2},a_{4},a_{3}$ is another sequence satisfying all the requirements, we conclude that $G\setminus a_{3}$ is sequentially $3$-rank-connected because we proved the statement for $i>3$.
\end{proof}
\begin{lemma}
\label{lem:fullyclosed}
Let $G$ be a sequentially $3$-rank-connected graph with $|V(G)|\geq 8$ and $a_{1},a_{2},\ldots, a_{k}$ be distinct vertices of $G$ such that $k\geq 4$, $k\neq |V(G)|-1$, and $\rho_{G}(\{a_{1},\ldots,a_{i}\})\leq 2$ for each $i\leq k$. If $\{a_{1},\ldots,a_{k}\}$ is a fully closed set of $G$, then there exists $i\in\{1,\ldots,k\}$ such that $G\setminus a_{i}$ or $G/a_{i}$ is sequentially $3$-rank-connected.
\end{lemma}
\begin{proof}
By Theorem~\ref{thm:bouchet} and Lemma~\ref{lem:gut}, we may assume that $k\neq |V(G)|$ and therefore $k\leq |V(G)|-2$.
Since $G$ is prime, we have $\rho_{G}(\{a_{1},\ldots,a_{k}\})=2$ and so, by Lemma~\ref{lem:getprime}, there is a vertex $a_{i}$ of $G$ such that $G\setminus a_{i}$ or $G/a_{i}$ is prime. By pivoting, we may assume that $G\setminus a_{i}$ is prime. Then, by Lemma~\ref{lem:gut}, $G\setminus a_{i}$ is sequentially $3$-rank-connected.
\end{proof}

\begin{proof}[Proof of Theorem~\ref{thm:main}]
By Proposition~\ref{prop:3rank_conn}, we may assume that $G$ is not $3$-rank-connected. So there is a subset $A$ of $V(G)$ such that $\rho_{G}(A)\leq 2$, $|A|\geq 3$, and $|V(G)-A|\geq 3$. If $G$ is internally $3$-rank-connected, then we may assume that $|A|=3$. By Lemma~\ref{lem:triplet}, we can assume that $A$ is a triplet of $G$ by pivoting. By Proposition~\ref{prop:internal}, there is a vertex $a\in A$ such that $G\setminus a$ is sequentially $3$-rank-connected. Hence we may assume that $G$ is not internally $3$-rank-connected.

Therefore, we may assume that $|A|\geq 4$ and $|V(G)-A|\geq 4$.
Since $G$ is sequentially $3$-rank-connected, $A$ or $V(G)-A$ is sequential in $G$. Therefore there exists a sequential set with at least $4$ elements. 

Let $X$ be a maximum sequential set of $G$. Then $X$ is a fully closed set of $G$. Furthermore, $|X|\neq |V(G)|-1$ because otherwise $V(G)$ is sequential in $G$.
Since $|X|\geq 4$, we conclude the proof by Lemma~\ref{lem:fullyclosed}.
\end{proof}

\paragraph{Acknowledgements}
The authors would like to thank the anonymous reviewers for their careful reading and useful comments. In particular, the paragraph following Theorem~\ref{thm:seq_matroid} was suggested by one of the anonymous reviewers.

\providecommand{\bysame}{\leavevmode\hbox to3em{\hrulefill}\thinspace}
\providecommand{\MR}{\relax\ifhmode\unskip\space\fi MR }
\providecommand{\MRhref}[2]{%
  \href{http://www.ams.org/mathscinet-getitem?mr=#1}{#2}
}
\providecommand{\href}[2]{#2}

\end{document}